\documentclass[a4paper]{amsart}
\usepackage{latexsym,amssymb,xcolor}
\usepackage[right]{lineno}
\usepackage[margin=30mm]{geometry}

\usepackage{hyperref}
\usepackage{cleveref}

\newtheorem{theorem}{Theorem}[section]
\newtheorem{lemma}[theorem]{Lemma}
\newtheorem{definition}[theorem]{Definition}
\newtheorem{corollary}[theorem]{Corollary}

\newtheorem{proposition}[theorem]{Proposition}
\newtheorem{remark}[theorem]{Remark}
\newtheorem{notation}[theorem]{Notation}

\newcommand{\Fix}{\mathop{\mathrm{Fix}}}
\newcommand{\OmSet}{\Omega_{\textrm{set}}}
\newcommand{\Irr}{\mathop{\mathrm{Irr}}}
\newcommand{\fpr}{\mathop{\mathrm{fpr}}\nolimits}
\newcommand{\GL}{\mathop{\mathrm{GL}}\nolimits}
\newcommand{\PGammaL}{\mathop{\mathrm{P\Gamma L}}\nolimits}
\newcommand{\PGammaU}{\mathop{\mathrm{P\Gamma U}}\nolimits}
\newcommand{\PGL}{\mathop{\mathrm{PGL}}\nolimits}
\newcommand{\PSL}{\mathop{\mathrm{PSL}}\nolimits}
\newcommand{\Aut}{\mathop{\mathrm{Aut}}}

\renewcommand{\O}{\mathop{\mathrm{O}}\nolimits}
\newcommand{\PSU}{\operatorname{\mathrm{PSU}}}
\newcommand{\POmega}{\mathop{\mathrm{P}\Omega}\nolimits}
\newcommand{\Alt}{\mathop{\mathrm{Alt}}\nolimits}
\newcommand{\Sym}{\mathop{\mathrm{Sym}}\nolimits}
\newcommand{\C}{\mathcal{C}}

\newcommand{\Sp}{\operatorname{\mathrm{Sp}}}
\newcommand{\PSp}{\mathop{\mathrm{PSp}}\nolimits}
\newcommand{\SGO}{\mathrm{S}(g,\Omega)}
\newcommand{\SGOp}{\mathrm{S}_1(g,\Omega)}
\newcommand{\SGOpp}{\mathrm{S}_2(g,\Omega)}
\renewcommand{\wr}{\mathop{\mathrm{wr}}}
\def\cent#1#2{{\bf C}_{{#1}}({{#2}})}
\def\Zent#1{{\bf Z}({{#1}})}
\def\norm#1#2{{\bf N}_{{{#1}}}({{{#2}}})}
\begin{document}

\title[Regular Orbits]{Finite primitive groups and regular orbits of group elements}

\author[S.~Guest]{Simon Guest}
\address{Simon Guest, Department of Mathematics, Imperial College London, \newline 
South Kensington Campus,
London
SW7 2AZ}\email{simon.guest@imperial.ac.uk}

\author[P. Spiga]{Pablo Spiga}
\address{Pablo Spiga, Dipartimento di Matematica e Applicazioni,\newline
University of Milano-Bicocca,
 Via Cozzi 55 Milano, MI 20125, Italy} \email{pablo.spiga@unimib.it}

\thanks{Address correspondence to P. Spiga,
E-mail: pablo.spiga@unimib.it}

\begin{abstract}
We prove that if $G$ is a finite primitive permutation group and if $g$ is an
element of $G$, then either $g$ has a cycle of length equal to its
order, or for some $r$, $m$ and $k$, the group $G \leq \Sym(m) \wr \Sym(r)$ preserves the
product structure of $r$ direct copies of the natural action of $\Sym(m)$ 
on $k$-sets. This gives an answer to a question of Siemons and Zalesski and a solution to a conjecture of Giudici, Praeger and the second author. 
\end{abstract}

\maketitle

\section{Introduction}\label{intro}
Let $G$ be a finite primitive permutation group on a set $\Omega$ and let $H \leq G$.
A \textit{regular orbit} of $H$ in $\Omega$ is an $H$-orbit of cardinality $|H|$. In particular, when $H=\langle g\rangle$,  a
regular orbit of $H$ is the point set of a $g$-cycle of length equal to the order $|g|$ of
$g$ in its disjoint cycle representation; we call such a cycle a \textit{regular cycle}. Motivated by a problem in representation theory, Siemons
and Zalesski~\cite{SZ} asked for conditions under which subgroups of primitive groups
could be guaranteed to have regular orbits.

Following the preliminary work of Emmett, Siemons and Zalesski~\cite{EZ,SZ,SZ1}, Giudici, Praeger and the second author~\cite{GPSreduction} have investigated this question for cyclic subgroups and have noticed that (surprisingly) primitive groups admitting elements with no regular orbits are very uncommon. 

The alternating groups $\Alt(m)$ (with $m \geq 7$) and the symmetric groups $\Sym(m)$ (with $m\geq 5$)  are
examples of groups admitting elements with no regular orbit. (For instance, the permutation
$g = (1\,2\,3)(4\,5)(6\,7) \in \Alt(m)$ has no cycle of length $|g| = 6$ in its natural action on $m$ points.) We denote the collection of the $k$-subsets of $\{1,\ldots,m\}$ by ${m\choose k}$ and we refer to the natural action of $\Sym(m)$ on ${m\choose k}$ as its
$k$-set action. It is easy to see that, for $k$ fixed and for  $m$ large compared to $k$, the group $\Sym(m)$ in its $k$-set action is another permutation group admitting elements with no regular orbit. Actually, it is possible to specify precisely the $k$-set actions of finite symmetric
groups for which all elements have regular cycles~\cite[Theorem~$1.1$]{GPSreduction}.

Loosely speaking the main result of this paper shows that the groups above are the  building blocks of every primitive permutation group admitting elements with no regular cycle.

\begin{theorem}\label{main} Let $G$ be a finite primitive permutation group on $\Omega$ and let $g\in G$. Then either $g$ has a regular cycle, or there exist integers $k\geq 1$, $r\geq 1$ and $m\geq 5$ such that $G$ preserves a product structure on $\Omega = \Delta^r$ and
$\Alt(m)^r\unlhd G\leq  \Sym(m) \wr \Sym(r)$, where  $\Delta = {m\choose k}$ and $\Sym(m)$ induces its natural $k$-set action on $\Delta$.
\end{theorem}

Our work gives a complete answer to the question: ``When do all elements of a finite primitive permutation group $G$ have
at least one regular cycle?'' In particular, we offer a precise answer to the question of Siemons and Zalesski~\cite{SZ} in the case of cyclic subgroups.

Theorem~\ref{main} was conjectured~\cite[Conjecture~$1.2$]{GPSreduction} by Giudici, Praeger and the second author. Using the O'Nan-Scott reduction theorem for finite primitive permutation groups, the authors of~\cite{GPSreduction} made a significant contribution towards a proof of Theorem~\ref{main}. In fact, in~\cite[Theorem~$1.3$]{GPSreduction} they have shown that Theorem~\ref{main} holds true, if it holds true for each almost simple primitive group with socle a non-abelian simple classical group. Therefore Theorem~\ref{main} follows from~\cite[Theorem~$1.3$]{GPSreduction} and the following.

\begin{theorem}\label{thrm:AS}
Let $G$ be a finite almost simple primitive group on $\Omega$ with socle $G_0$ a non-abelian simple classical group, and let $g \in G$. Then either $g$ has a regular cycle, or $(G_0,\Omega)$ is isomorphic to one of the following $(\Alt(5),\{1,\ldots,5\})$, $(\Alt(6),\{1,\ldots,6\})$ or $(\Alt(8),\{1,\ldots,8\})$.
\end{theorem}

Theorem~\ref{main} is yet another result showing that the primitive group $\Sym(m)\wr\Sym(r)$ endowed with its natural product action has rather peculiar properties (see for example~\cite[Section~$6$]{Cameron1}). Since it is computationally quite cheap to determine the order and the cycle lengths of a permutation, Theorem~\ref{main} has some practical interest and can be easily implemented as part of a recognition algorithm for primitive permutation groups~\cite{magma}. In practice Theorem~\ref{main} can be used to determine how much effort 
should be invested into running the ``Jellyfish algorithm''~\cite{LNPS,NeSe} on a primitive group.

Theorem~\ref{main} has also the following striking corollary.
\begin{corollary}\label{cor}
 Let $G$ be a finite primitive permutation group on $\Omega$ and let $g\in G$. Then either $|g|\leq |\Omega|$, or there exist integers $k\geq 1$, $r\geq 1$ and $m\geq 5$ such that $G$ preserves a product structure on $\Omega = \Delta^r$ and
$\Alt(m)^r\unlhd G\leq  \Sym(m) \wr \Sym(r)$, where  $\Delta = {m\choose k}$ and $\Sym(m)$ induces its natural $k$-set action on $\Delta$.
\end{corollary}

\section{Preliminary results}\label{preliminaries}
We start our analysis with some elementary (but very useful) lemmas.

\begin{lemma}[{{\cite[Lemma~$2.1$]{GPSreduction}}}]\label{basic}
Let $G$ be a permutation group on $\Omega$. Assume that for every
$g\in G$, with $|g|$ square-free, $g$ has a regular cycle. Then, for every $g\in G$, the element $g$ has a regular cycle.
\end{lemma}

In the proof of Theorem~\ref{thrm:AS}, the previous lemma allows us to consider only elements of square-free
order. Before proceeding, we introduce a definition.

\begin{definition}{\rm Let $G$ be a finite permutation group on the finite set $\Omega$ and let $x \in G$. We let
$\Fix_\Omega(x)=\{\omega\in\Omega\mid \omega^x=\omega\}$ and
 $\fpr_\Omega(x)=|\Fix_\Omega(x)|/|\Omega|$. (The label ``fpr'' stands for ``fixed-point-ratio''.)}
\end{definition}

 We recall a basic lemma on fixed point ratios.

\begin{lemma}[{{\cite[Lemma~$2.5$]{LS}}}]\label{obvious3}
Let $G$ be a transitive permutation group on $\Omega$, let $H$ be a point stabilizer and let $g \in G$.
Then \[\mathrm{fpr}_\Omega(g)=\frac{|g^G\cap H|}{|g^G|}.\]
\end{lemma}

The next lemma is one of the main tools used in this paper.

\begin{lemma}[{{\cite[Lemma~$2.4$]{GPSreduction}}}]\label{lemma:apeman}Let $G$ be a transitive permutation group on $\Omega$ and let $g$ be in $G\setminus\{1\}$. The element
 $g$ has a regular cycle if and only if
\[\bigcup_{\substack{r\mid |g|\\r\,\mathrm{
 prime}}}\mathrm{Fix}_\Omega(g^{|g|/r})\subsetneq\Omega,\]
In particular, if \[\sum_{\substack{r\mid |g|\\r\, \mathrm{
 prime}}}\mathrm{fpr}_\Omega(g^{|g|/r})<1,\]
then $g$ has a regular cycle.
\end{lemma}

In light of Lemma~\ref{lemma:apeman}, we define $\SGO$ to be
\begin{equation*} 
  \SGO:= \sum_{\substack{r\mid |g|\\r\, \mathrm{
 prime}}}\mathrm{fpr}_\Omega(g^{|g|/r})
\end{equation*}
and in order to prove Theorem~\ref{thrm:AS}, we will  show that (in most cases) $\SGO<1$ for every $g \in G\setminus\{1\}$. 

Unless specified otherwise, for the rest of this paper we will use the following notation.
\begin{notation}\label{generalnotation}{\rm We let $G$ denote a finite almost simple classical group defined over the finite field of size $q$ and  with socle $G_0$. For twisted groups our notation for $q$ is such that $\PSU_n(q)$ and $\POmega_{2m}^-(q)$ are the twisted groups contained in $\PSL_n(q^2)$ and $\POmega_{2m}^+(q^2)$, respectively.

We write $q=p^e$, for some prime $p$ and some $e\geq 1$, and we define
\begin{equation*}
q_0:= \begin{cases}
q^{2} & \mbox{if $G$ is unitary}\\
q & \mbox{otherwise.}
\end{cases}
\end{equation*}

We let $V$ be the natural module defined over the field $\mathbb{F}_{q_0}$ of size $q_0$ for the covering group of $G_0$, and we let $n$ be the dimension of $V$ over $\mathbb{F}_{q_0}$. Finally, we denote by $m$ the Witt index of $G_0$, that is, the maximal dimension of a totally singular subspace of $V$.}
\end{notation}

We now recall a fundamental and pioneering result of Liebeck and Saxl~\cite[Theorem~$1$]{LS}:
\begin{theorem}\label{LSthm}
Let $G$ and $q$ be as in Notation~$\ref{generalnotation}$. If $G$ acts transitively and faithfully on a set $\Omega$, then for every $x\in G\setminus\{1\}$ we have
\[\mathrm{fpr}_\Omega(x)\leq\frac{4}{3q}\]
apart from the exceptions listed in~\cite[Table~$1$]{LS}.
\end{theorem}

\begin{remark}\label{remark1}{\rm
Unfortunately (for our application) the exceptions in~\cite[Table~$1$]{LS} are all genuine. For each such case, however,~\cite[Table~$1$]{LS} lists all possible $G$-sets $\Omega$ and all possible group elements $x$ with $\fpr_\Omega(x)>4/(3q)$. Furthermore, the third column of~\cite[Table~$1$]{LS} gives an exact formula for $\fpr_\Omega(x)$ in each case.}
\end{remark}

\subsection{Number theory}
We recall that in elementary number theory, for a positive integer $n$, the number of
distinct prime divisors of $n$ is denoted by $\omega(n)$. We will find the following result of Robin useful.

\begin{lemma}[{\cite[Theorem~13]{Robin}}]\label{robin}For $n \geq 26$, we have $\omega(n)\leq \log(n)/(\log(\log(n))- 1.1714)$.
\end{lemma}
The next lemma is much more elementary and we omit its proof.
\begin{lemma}\label{series}
For $q\geq 2$, we have \[\sum_{\ell= 0}^{\infty}\ell q^{-\ell}=\frac{q}{(q-1)^2}.\]
\end{lemma}

\section{Classical groups of small dimension}

We start by introducing some more notation that will prove useful for the proof of Theorems~\ref{smalldimension} and~\ref{thrmnonsubspaceactions}. Let $|\Aut(G_0)|_{p'}$ denote the $p'$-part of  $|\Aut(G_0)|$ and set
\begin{equation*}
a(n,q):=1+\frac{\log(|\Aut(G_0)|_{p'})}{\log(\log(|\Aut(G_0)|_{p'})-1.1714)}.
\end{equation*}
Comprehensive information on $|G_0|$ and $|A|$ can be found in~\cite[Table~$5$, page~xvi]{ATLAS}. 
Observe that, once the Lie type of $G_0$ is fixed, $a(n,q)$ is indeed an explicit function of $n$ and $q$. Suppose that $G_0\ncong \PSL_2(4)$,  $\PSL_2(7)$,  and observe that $|\Aut(G_0)|_{p'}\geq 26$. By  Lemma~\ref{robin}, we have
\begin{equation}\label{eq:2s}
\omega(|\Aut(G_0)|)= 1+\omega(|\Aut(G_0)|_{p'})\leq 1+\frac{\log(|\Aut(G_0)|_{p'})}{\log(\log(|\Aut(G_0)|_{p'})-1.1714)}=a(n,q),
\end{equation}
where the summand $1$ corresponds to the prime divisor $p$ of $|\Aut(G_0)|$.

Now we prove a preliminary theorem that will simplify some of our arguments later and  will also provide a taste of the arguments used in the rest of the paper.
\begin{theorem}\label{smalldimension}
Let $G,G_0$ and $n$ be as in Notation~$\ref{generalnotation}$ with $n\leq 4$. Suppose that $G$ acts primitively on $\Omega$ and let $g\in G$. Then either $g$ has a regular cycle, or $(G_0,\Omega)$ is isomorphic to $(\Alt(5),
\{1,\ldots,5\})$, $(\Alt(6),\{1,\ldots,6\})$ or $(\Alt(8),\{1,\ldots,8\})$.
\end{theorem}
\begin{proof}
Using the various isomorphisms between classical groups,  
we may assume that  $G_0$ is either $\PSL_2(q)$ with $q\geq 5$, or $\PSL_3(q)$ with $q\geq 3$, or $\PSL_4(q)$ with $q\geq 2$, or $\PSU_3(q)$ with $q\geq 3$, or $\PSU_4(q)$ with $q\geq 2$, or $\PSp_4(q)$ with $q\geq 4$. 

We start by explaining our strategy. Let $A = \Aut(G_0)$ so that, in particular, we may assume that $G_0\leq G\leq A$. 

For now we 
suppose that $G_0 \not \cong \PSL_2(q), \, \PSL_4(2)$, or $\PSU_4(2)$. In particular, from~\cite[Table~$1$]{LS}, we have $\fpr_\Omega(g^{|g|/r})\leq 4/(3q)$, for each prime $r$ dividing $|g|$. Thus from~\eqref{eq:2s} we get
\begin{eqnarray}\label{eq:neverend}
\SGO \leq \sum_{\substack{ r \mid |g| \\ r \, \text{prime}} }\frac{4}{3q}=\omega(|g|)\frac{4}{3q}\leq \omega(|A|)\frac{4}{3q}\leq a(n,q)\frac{4}{3q}.
\end{eqnarray}
In particular, if $4a(n,q)<3q$, then Lemma~\ref{basic} implies that $g$ has a regular cycle on $\Omega$. Now we check with a computer (implementing the function $a(n,q)$)  that there are only a finite number of pairs $(n,q)$ such that $4a(n,q)\geq 3q$. For each of the remaining groups, we compute the exact value of $\omega(|A|)$ and show that either $4\omega(|A|)<3q$ or $G_0$ is one of the groups in Table~\ref{tableexpsmall}.

When $G_0 \cong \PSL_2(q)$, $\PSL_4(2)$, or $\PSU_4(2)$, we substitute the upper bound $4/(3q)$ with the amended upper bound in~\cite[Table~$1$]{LS} and follow the calculation above. For each of these groups, we show that either $\SGO<1$ or $G_0$ is one of the groups in Table~\ref{tableexpsmall}.

The remaining groups $G_0$ in Table~\ref{tableexpsmall} are rather small and we can compute the exact value of $\max\{\omega(|g|)\mid g\in \Aut(G_0)\}$. Using this value in~\eqref{eq:neverend}  (in place of $a(n,q)$) and using the upper bound for $\fpr_\Omega(g^{|g|/r})$ in~\cite[Theorem~$1$]{LS}, we have $\SGO < 1$ and $g$ has a regular orbit,  or $G_0$ is isomorphic to one of $\PSL_2(5)$, $\PSL_2(7)$, $\PSL_2(9)$, $\PSL_4(2)$, $\PSL_4(4)$, $\PSU_4(2)$, $\PSU_4(4)$.

Finally, for each group $G_0$ in this short list, we compute for every $G$ with $G_0\leq G\leq \Aut(G_0)$ all the primitive permutation representations of $G$ and check whether every conjugacy class representative has a regular orbit: the only exceptions are the examples in the statement of this theorem. 
\end{proof}

\begin{table}
\begin{tabular}{|c|l|}\hline
Group $G_0$&Comments\\\hline
$\PSL_n(q)$&$(n,q)\in \{(2,5),(2,7),(2,8),(2,9),(2,11),(2,16),(2,19),$\\
&$\qquad (3,3),(3,4),(3,5),(4,2),(4,3),(4,4),(4,5),(4,8)\}$\\
$\PSU_n(q)$&$(n,q)\in \{(3,3),(3,4),(3,5),(4,2),(4,3),(4,4),(4,5),(4,8)\}$\\
$\PSp_{2m}(q)$&$(m,q)\in \{(2,4),(2,5)\}$\\\hline
\end{tabular}
\caption{Exceptional cases in the proof of~\Cref{smalldimension}}\label{tableexpsmall}
\end{table} 

\section{Non-subspace actions of classical groups}\label{non-subspace}
In studying actions of classical groups, it is rather natural to distinguish between those actions which permute the subspaces of the natural module and those which do not. The stabilizers of subspaces are generally rather large (every parabolic subgroup falls into this class) and therefore the fixed-point-ratio in these cases also tends to be rather large. As the culmination of an important series of papers~\cite{Tim1,Tim2,Tim3,Tim4}, Burness obtained remarkably good upper bounds on the fixed-point-ratio for each finite almost simple classical group in \textit{non-subspace actions}. For future reference, we first need to make the definition of non-subspace action precise.

\begin{definition}\label{def}{\rm Assume Notation~$\ref{generalnotation}$. A subgroup $H$ of $G$ is a \textit{subspace subgroup} if, for each maximal subgroup $M$ of $G_0$ that contains $H\cap G_0$,  one of the following conditions hold:
\begin{description}
\item[(a)]$M$ is the stabilizer in $G_0$ of a proper non-zero subspace $U$ of $V$, where $U$ is totally singular, or non-degenerate, or, if $G_0$ is orthogonal and $p=2$, a non-singular $1$-subspace ($U$ can be any subspace if $G_0=\PSL(V)$);
\item[(b)]$M=\O_{2m}^{\pm}(q)$ and $(G_0,p)=(\Sp_{2m}(q)',2)$.
\end{description}}
\end{definition}
A transitive action of $G$ on a set $\Omega$ is a \textit{subspace action} if the point stabilizer $G_\omega$ of the element $\omega\in \Omega$ is a subspace subgroup of $G$; \textit{non-subspace subgroups} and \textit{actions} are defined accordingly.

For the convenience of the reader we report~\cite[Theorem~$1$]{Tim1}. 
\begin{theorem}[{\cite[Theorem~$1$]{Tim1}}]\label{timthm}
Let $G$ be a finite almost simple classical group acting transitively and faithfully on a set $\Omega$ with point stabilizer $G_\omega\leq H$, where $H$ is a maximal non-subspace subgroup of $G$. Let $G_0$ be the socle of $G$. Then, for every $x\in G\setminus\{1\}$, we have \[\mathrm{fpr}_\Omega(x)<|x^G|^{-\frac{1}{2}+\frac{1}{n}+\iota},\]
 where $x^G$ is the conjugacy class of $x$ in $G$, $n$ is as defined in~\cite[Definition~2]{Tim1} and either $\iota=0$ or $(G_0,H,\iota)$ is listed in~\cite[Table~1]{Tim1}.
\end{theorem}
In Theorem~\ref{timthm}, apart from the two rather natural exceptions $\PSL_3(2)$ and $\PSp_4(2)'$ (where $n=2$), $n$ is exactly as in Notation~$\ref{generalnotation}$.

The upper bound in Theorem~\ref{timthm} is quite sharp when $n$ is large and is extremely useful for our application. However, for small values of $n$,  Theorem~\ref{timthm} loses all of its power. For instance, when $G_0=\PSL_4(q)$, we see from~\cite[Table~$1$]{Tim1} that $\iota=1/4$ and hence the upper bound in Theorem~\ref{timthm} only says $\fpr_\Omega(x)\leq |x^G|^0=1$. In these cases, we will often revert to the bounds in Theorem~\ref{LSthm}. However, we point out that recently Burness and the first author~\cite{BG} have strengthened Theorem~\ref{timthm} for linear groups of very small rank. 

We are now ready to make our first important contribution towards the proof of Theorem~\ref{thrm:AS}. 
\begin{theorem}\label{thrmnonsubspaceactions}
Assume Notation~$\ref{generalnotation}$. Let $G$ be acting primitively with a non-subspace action on a set $\Omega$ and let $g\in G$. Then either $g$ has a regular cycle, or $(G_0,\Omega)$ is isomorphic to  $(\Alt(5),\{1,\ldots,5\})$, $(\Alt(6),\{1,\ldots,6\})$ or $(\Alt(8),\{1,\ldots,8\})$.
\end{theorem}
\begin{proof}
As in Theorem~\ref{smalldimension}, the proof  makes heavy use of a computer. 
Observe that by Theorem~\ref{smalldimension}, we may assume that $n \geq 5$.

Let $m(G_0)$ denote the minimal degree of a non-trivial transitive permutation representation of $G_0$. Clearly, for every $x\in G\setminus \{1\}$, we have
\begin{equation*}
|x^{G}|=|G:\cent G x|\geq |G_0\cent G x:\cent G x|=|G_0:\cent {{G_0}}x|\geq m(G_0).
\end{equation*}
The explicit value for $m(G_0)$ (as  a function of $n$ and $q$) can be found in ~\cite[Table
~$5.2.A$]{KL}. (We note that the rows
corresponding to the classical groups $\POmega_{2m}^+(q)$ and
$\PSU_{2m}(2)$ in~\cite[Table~$5.2.A$]{KL} are incorrect as brilliantly observed by Vasil$'$ev and Mazurov in~\cite{Vav4}. For these rows we refer to the amended values in~\cite{Vav4}. A table listing $m(G_0)$ for each non-abelian simple group and taking into account all the known corrections can be found in~\cite{GMPS}).

Denote by $t(G_0)$ the upper bound for $\fpr_\Omega(x)$ obtained in Theorem~\ref{LSthm} (so, typically $t(G_0)=4/(3q)$ and the exceptional cases, together with the corresponding amended upper bound, are found in~\cite[Table~$1$]{LS}; see Remark~\ref{remark1}). Let
\begin{equation*}
  t(n,q):=\min\left\{t(G_0),m(G_0)^{-\frac{1}{2}+\frac{1}{n}+\iota}\right\},
\end{equation*}
and observe that, for each almost simple classical group, $t(n,q)$ is an explicit function of $n$ and $q$.

Let $A = \Aut(G_0)$. By definition of $t(n,q)$ and Theorems~\ref{LSthm} and~\ref{timthm}, we have $\fpr_\Omega(g^{|g|/r})\leq t(n,q)$ for each prime divisor $r$ of $|g|$. Thus by~\eqref{eq:2s} we get
\begin{eqnarray}\label{eq:refrefref}
\SGO &\leq& \sum_{\substack{ r \mid |g| \\ r \, \text{prime}}}t(n,q)=\omega(|g|)t(n,q)\leq \omega(|A|)t(n,q)\leq a(n,q)t(n,q).
\end{eqnarray}
In particular, if $a(n,q)t(n,q)<1$, then Lemma~\ref{basic} implies that $g$ has a regular cycle  in its action on $\Omega$. Using a computer we find the pairs $(n,q)$ such that $a(n,q)t(n,q)\geq 1$. For each of these remaining groups, we compute the exact value of $\omega(|A|)$ and we find that either $\omega(|A|)t(n,q) <1$ or $G_0$ is one of the groups in Table~\ref{tableexp}.

\begin{table}
\begin{tabular}{|c|l|}\hline
Group $G_0$&Comments\\\hline
$\PSL_n(q)$&$(n,q)\in \{(5,2),(5,3),(5,4),(6,2),(6,3),(6,4),(6,5),(6,7),(6,8),$\\
&$\qquad (6,9),(6,11),(7,2),(8,2),(8,3),(10,2)\}$\\
$\PSU_n(q)$&$(n,q)\in \{ (6,2),(6,3)\}$\\
$\PSp_{2m}(q)$&$(m,q)\in \{(3,2),(3,3),(3,4),(3,5),(3,7),(3,8),(3,9),(4,2),(5,2)\}$\\
$\POmega_{2m+1}(q)$&$(m,q)=(3,3)$\\
$\POmega_{2m}^+(q)$&$(m,q)\in \{(4,2),(4,3),(5,2)\}$\\
$\POmega_{2m}^-(q)$&$(m,q)\in \{(4,2),(4,3),(4,4),(5,2)\}$\\\hline
\end{tabular}
\caption{Exceptional cases in the proof of~\Cref{thrmnonsubspaceactions}}\label{tableexp}
\end{table} 

Using \texttt{magma} \cite{magma}, for each  $G_0$ in Table~\ref{tableexp}, we can compute $\max\{\omega(|g|)\mid g\in \Aut(G_0)\}$. By using this value as an upper bound for $\omega(|g|)$ in~\eqref{eq:refrefref} and by arguing as in the previous paragraph, we see that either $g$ has a regular orbit on $\Omega$ or $G_0$ is one of the following groups
\begin{align*}
&\mathrm{PSL}_6(2),\,\mathrm{PSL}_6(3),\,\mathrm{PSL}_6(4),\,\mathrm{PSL}_6(5),\,\mathrm{PSU}_6(2),\,
\mathrm{PSp}_6(2),\,\mathrm{PSp}_6(3),\,\mathrm{P}\Omega_8^+(2),\,\mathrm{P}\Omega_8^+(3),\,\mathrm{P}\Omega_8^-(2).
\end{align*}

For $G_0 = \POmega_8^{+}(3)$ we use the full strength of Theorem~\ref{timthm}: for each prime divisor $r$ of $|g|$ we have $\fpr_{\Omega}(g^{|g|/r}) \le |x_r^G|^{-1/2+1/n+\iota}$ where $x_r:=g^{|g|/r}$. In particular, using~\cite[Table~1]{Tim1}, for any non-subspace action we have $\iota \le 0.219$ and thus $\fpr_{\Omega}(g^{|g|/r}) \le |x_r^G|^{-0.156}$. Let $\alpha(g,r) := |x_r^G|^{-0.156}$. We can check in \texttt{magma} that for every almost simple $G$ with socle $\POmega_8^{+}(3)$, and every $g\in G$ of square-free order, the  sum $\sum \alpha(g,r)$ (over all prime divisors $r$ of $|g|$)  is less than $1$. Therefore the case $G_0= \POmega^{+}_8(3)$ is eliminated. The group $G_0 = \PSL_6(5)$ is dealt with in exactly the same fashion.

Finally, for each remaining group $G_0$, we construct in \texttt{magma} all of  the primitive permutation representations of every almost simple group $G$ with socle $G_0$ and check directly that every conjugacy class representative has a regular cycle. 
\end{proof}

\section{A reduction for subspace actions}\label{other reduction}

In view of Theorems~\ref{smalldimension} and~\ref{thrmnonsubspaceactions} we may now consider only subspace actions of almost simple classical groups $G$  and we may assume that $n\geq 5$.

\begin{remark}\label{remrem2}{\rm 
Let $G$ be an almost simple group as above acting on the set $\Omega$ and let $\alpha\in \Omega$. Now there are two main cases to consider: either $G_\alpha\cap G_0$ is a maximal subgroup of $G_0$, or $G_\alpha\cap G_0$ is not maximal in $G_0$. In the first case, $G_\alpha\cap G_0$ is one of the subgroups given in Definition~\ref{def} and hence we may identify the action of $G$ on $\Omega$ with the natural extension to $G$ of the action of $G_0$ on the coset space $G_0/(G_\alpha\cap G_0)$. (We give an example of this case: $G=\PGL_n(q)$, $G_0=\PSL_n(q)$ and $G_\alpha\cap G_0$ is a maximal parabolic subgroup of $G_0$ stabilizing a $k$-dimensional subspace of the natural module $V=\mathbb{F}_q^n$. Clearly, we may identify the action of $G$ on $\Omega$ with the action of $G$ on the $k$-dimensional subspaces of $V$.) In particular, in this first case the action of $G$ is perfectly understood. 

Now suppose that $G_\alpha\cap G_0$ is not maximal in $G_0$ and let $M$ be a maximal subgroup of $G_0$ with $G_\alpha\cap G_0<M$. Clearly $\norm {{G}}{{G_\alpha\cap G_0}}=G_\alpha\nleq \norm G M$ and hence (using the terminology in~\cite[page~$66$]{KL}) the subgroup $G_\alpha\cap G_0$ is a $G$-\textit{novelty}. All such novelties were classified for $n\geq 13$ in~\cite[Table~$3.5$H]{KL}; we will make use of this table to identify the group $G_\alpha$ and the action of $G$ on $\Omega$.}
\end{remark}

Next, we state a definition, which will play a crucial role in this paper.
\begin{definition}\label{def:char}{\rm
Let $G$ be a finite group, and let $\Irr(G)$ denote the set of complex irreducible characters of $G$ and let $\langle\,,\,\rangle_G$ denote the natural inner-product on the class functions of $G$. Given complex characters $\pi$ and $\pi'$  of $G$, we write $\pi\leq \pi'$ if, for each $\eta\in \Irr(G)$, we have $\langle \pi,\eta\rangle_G\leq \langle \pi',\eta\rangle_G$. In other words, writing  \[\pi=\sum_{\eta\in \Irr(G)}a_\eta \eta\quad\textrm{and}\quad \pi'=\sum_{\eta\in \Irr(G)}a_\eta'\eta,\]  we have $\pi \le \pi'$ if and only if  $a_\eta\leq a_\eta'$ for every $\eta \in \Irr(G)$.
}
\end{definition}

Given a subgroup $H$ of $G$ and a character $\pi$ of $G$, we write $\pi_{|H}$ for the restriction of $\pi$ to $H$.

\begin{lemma}\label{characterlemma}Let $G$ be a finite group with two permutation representations on the finite sets $\Omega_1$ and $\Omega_2$. For $i\in \{1,2\}$, let $\pi_{i}$ be the permutation character for the action of $G$ on $\Omega_i$, and suppose that $\pi_{1}\leq \pi_{2}$. If $g \in G$ has $\ell$ cycles of length $|g|$ on $\Omega_1$, then $g$ has at least $\ell$ cycles of length $|g|$  on $\Omega_2$.
\end{lemma}
\begin{proof}
The proof of this result is somehow implicit in~\cite[Theorem~$3.4$]{SZ}. Let $g\in G$, let $\ell$ be the number of cycles of $g$ of length $|g|$ on $\Omega_1$ and let $\rho$ be the regular character of the cyclic group $\langle g\rangle$.

By~\cite[Corollary~$2.5$]{SZ}, we see that $\ell\rho\leq (\pi_1)_{|\langle g\rangle}$. Since $\pi_1\leq \pi_2$, we have $\ell\rho\leq (\pi_2)_{|\langle g\rangle}$. Now another application of~\cite[Corollary~$2.5$]{SZ} implies that $g$ has at least $\ell$ orbits of length $|g|$ on $\Omega_2$.
\end{proof}

In a nutshell, Lemma~\ref{characterlemma} shows that in order to prove Theorem~\ref{thrm:AS} we need only consider those actions that have  permutation character that is minimal with respect to the partial ordering in Definition~\ref{def:char}.

In the following, we use the term ``totally singular'' instead of separating into ``totally isotropic'' or ``totally singular'' subspaces according to the type of $V$. We use  the term ``non-singular'' to mean having zero radical. For the benefit of the reader we report here~\cite[Corollary~$1.2$]{FGK}.
\begin{theorem}[{\cite[Corollary~$1.2$]{FGK}}]\label{FGK}
Let $G$ be an almost simple classical group with natural
module $V$ of dimension $n$. If $G$ is linear, assume that $G$ does not contain a
graph automorphism. Let $k\in \{2, \ldots, n-1\}$, and let $K$ be the stabilizer of a non-degenerate or totally singular $k$-subspace of $V$. Let $P$ be the stabilizer of a
singular $1$-subspace of $V$. Then $1_P^G\leq 1_K^G$
unless one of the following holds.
\begin{description}
\item[(1)] $K$ is the stabilizer of a maximal totally singular subspace of $V$.
\item[(2)] $K$ is the stabilizer of a $2$-subspace containing no singular points or the
orthogonal complement of such. 
\end{description}
\end{theorem}

Theorems~\ref{thrmnonsubspaceactions},~\ref{FGK}, Lemma~\ref{characterlemma} and Remark~\ref{remrem2} imply the following reduction.
\begin{proposition} \label{reduction}
Assume Notation~$\ref{generalnotation}$. Theorem~$\ref{thrm:AS}$ holds if it holds for the following cases of subspace $G$-actions on $\Omega$:
\begin{description}
\item[Case~$(i)$]$\Omega$ consists of totally singular $1$-subspaces; 
\item[Case~$(ii)$]$\Omega$ consists of non-degenerate $1$-subspaces; 
\item[Case~$(iii)$]$\Omega$ consists of  maximal totally singular subspaces; 
\item[Case~$(iv)$]$\Omega$ consists of non-degenerate $2$-subspaces containing no singular points; 
\item[Case~$(v)$]$G_0=\PSL_n(q)$ and $G$ contains a graph automorphism; 
\item[Case~$(vi)$]$G_0=\Sp_{n}(q)'$ with $q$ even, and the stabilizer  in $G_0$ of a point is isomorphic to $\O_{n}^\pm(q)$; 
\item[Case~$(vii)$]$G_\alpha\cap G_0$ is a $G$-novelty and $G_\alpha$ is a subspace subgroup, where $G_\alpha$ is the stabilizer of the point $\alpha$.
\end{description}
 \end{proposition}

\begin{remark}\label{smallorth}{\rm 
As $\POmega_3(q)\cong \PSL_2(q)$ and $\POmega_5(q)\cong \PSp_4(q)$, from Theorem~\ref{smalldimension} we may assume in the rest of this paper that $n\geq 7$ when $G_0=\POmega_n(q)$. Similarly, as $\POmega_4^+(q)\cong \PSL_2(q)\times \PSL_2(q)$, $\POmega_6^+(q)\cong \PSL_4(q)$, $\POmega_4^-(q)\cong \PSL_2(q^2)$ and $\POmega_6^-(q)\cong \PSU_4(q)$, we may assume that $n\geq 8$ when $G_0=\POmega_n^\varepsilon(q)$ with $\varepsilon\in \{+,-\}$.}
\end{remark}

\begin{remark}\label{remarkorbits}{\rm We recall a few standard facts on the subspace actions in \textbf{Cases }(i)--(iii), for some more detailed information see for example~\cite{Cameron}. The group $G_0$ acts primitively on both the totally singular $1$-subspaces and the maximal totally singular subspaces of $V$. In particular, in \textbf{Case }(i) and (iii), the set $\Omega$ is the whole set of totally singular $1$-subspaces and the whole set of maximal totally singular subspaces, respectively. The same comment applies in \textbf{Case }(ii) except when $G_0$ is orthogonal and $q$ is odd. When $G_0$ is orthogonal and $q$ is odd, $G_0$ has two orbits on non-degenerate $1$-subspaces. Following~\cite[Section~$3.2$]{GK}, we denote these two $G_0$-orbits by ${\bf NS_1^+}$ and ${\bf NS_1^-}$, and we write ${\bf NS_1}={\bf NS_1^+}\cup {\bf NS_1^-}$. If $Q$ is the quadratic form on $V$ for $G_0$, then ${\bf NS_1^+}=\{\langle v\rangle\mid v\in V\setminus\{0\},\,Q(v) \textrm{ is a square in }  \mathbb{F}_{q_0}\}$ and ${\bf NS_1^-}=\{\langle v\rangle\mid v\in V\setminus\{0\},\,Q(v) \textrm{ is not a square in } \mathbb{F}_{q_0}\}$. Moreover,
\[
|{\bf NS_1^+}|=
\begin{cases}
q^{(n-1)/2}(q^{(n-1)/2}+1)/2&\textrm{if } G_0=\POmega_n(q),\,nq\textrm{ odd},\\
q^{n/2-1}(q^{n/2}-1)/2&\textrm{if }G_0=\POmega_n^+(q),\,n\textrm{ even,}\,q\textrm{ odd},\\
q^{n/2-1}(q^{n/2}+1)/2&\textrm{if }G_0=\POmega_n^-(q),\,n\textrm{ even,}\,q\textrm{ odd},\\
\end{cases}
\]
\[
|{\bf NS_1^-}|=
\begin{cases}
q^{(n-1)/2}(q^{(n-1)/2}-1)/2&\textrm{if } G_0=\POmega_n(q),\,nq\textrm{ odd},\\
q^{n/2-1}(q^{n/2}-1)/2&\textrm{if }G_0=\POmega_n^+(q),\,n\textrm{ even,}\,q\textrm{ odd},\\
q^{n/2-1}(q^{n/2}+1)/2&\textrm{if }G_0=\POmega_n^-(q),\,n\textrm{ even,}\,q\textrm{ odd}.\\
\end{cases}
\]
When $n$ is odd, $G_0=\POmega_{n}(q)$ and $\eta\in\{+,-\}$, the stabilizer of a point in ${\bf NS_1^\eta}$ is  isomorphic to $\O_1(q)\perp\O_{2m}^\eta(q)$ (these points are usually referred to as $\eta$ points, as in~\cite{ATLAS}). When $n$ is even,
 $G_0=\POmega_n^\varepsilon(q)$, $\varepsilon\in \{+,-\}$ and $q$ is odd, we have $|{\bf NS_1^+}|=|{\bf NS_1^-}|$ and  every element of
 $\mathrm{PGO}_n^\varepsilon(q)\setminus\mathrm{PSO}_n^\varepsilon(q)$ interchanges these two $G_0$-orbits. In particular, in this latter case we must
 have $G\leq \mathrm{P\Gamma O}^\varepsilon_{n}(q)$ otherwise the stabilizer of a point is not maximal and we do not have a primitive action.}
\end{remark}

Recall the definition of $q_0$ in Notation~\ref{generalnotation} and let
\begin{align*} 
 \SGOp:= \sum_{\substack{r\mid |g|\\ r \, \text{prime} \\ r\mid ep(q_0-1)}}
\mathrm{fpr}_\Omega(g^{|g|/r})  & \quad\mbox{ and } &\SGOpp:=  \sum_{\substack{r\mid |g|\\ r \, \text{prime} \\ r\nmid ep(q_0-1)}}
\mathrm{fpr}_\Omega(g^{|g|/r})
\end{align*}
so that 
\begin{equation*} 
  \SGO = \sum_{\substack{r\mid |g|\\ r \, \text{prime}}}
\mathrm{fpr}_\Omega(g^{|g|/r}) = \SGOp +\SGOpp.
\end{equation*}

For some of the $G$-sets in Proposition~\ref{reduction}, it will be useful to have stronger upper bounds on fixed-point-ratios for elements of certain prime orders. 

 Suppose that $r$ is a prime  with $r\nmid pe(q_0-1)$. 
Let $\ell$ be the minimum positive integer with $r\mid (q_0^\ell-1)$ and let $x$ be an element of $G$ of order $r$. Clearly, $\ell\geq 2$. Moreover, $x$ lies in the inner-diagonal subgroup of $G$ and  is a semisimple. Let $x'$ be a semisimple element of $\GL(V)$ projecting to $x$. Since $x'$ is semisimple, by Maschke's theorem the vector space $V$ can be written as a direct sum $V=[V,x']\oplus \cent V {x'}$, where $[V,x']=\{-v+vx'\mid v\in V\}$. Let  $\ell':= \dim_{\mathbb{F}_{q_0}}[V,x']$ and observe that $\ell' \geq \ell$ because every non-trivial irreducible $\mathbb{F}_{q_0}\langle x'\rangle$-submodule of $V$ has dimension $\ell$.

\begin{proposition} \label{fprell}
  Suppose that $G$ and $\Omega$ are as  in Proposition~$\ref{reduction}$ \textrm{{\bf Cases }}(i), (ii). Let $x \in G$ have prime order $r$ with $r\nmid ep(q_0-1)$, and let  $\ell$ and $\ell'$ be as above.  Then 
  \[ \mathrm{fpr}_\Omega(x) \le \begin{cases}
    1/q_0^{\ell'} & \mbox{in \textrm{{\bf Case}} (i) with $G$  linear;}\\
 2/q_0^{\ell'} & \mbox{in \textrm{{\bf Case}} (i) with $G$  unitary or orthogonal;}\\
   2/q_0^{\ell'} & \mbox{in \textrm{{\bf Case}} (ii) with $G$  unitary;}\\
   36/(13q_0^{\ell'}) & \mbox{in \textrm{{\bf Case}} (ii) with $G$  orthogonal and $n\geq 7$.}
  \end{cases} \]
 \end{proposition}
 \begin{proof} 
Let $v\in V\setminus\{0\}$ with $\langle v\rangle^x=\langle v\rangle$. We can write $v=w+t$, with $w\in [V,x']$ and $t\in \cent V {x'}$. Since $\langle v\rangle^x=\langle v\rangle$, we have $vx'=\lambda v$, for some $\lambda\in \mathbb{F}_{q_0}$. Since $[V,x']$ is $x'$-invariant, we obtain $wx'=\lambda w$. Thus $\langle w\rangle$ is a $1$-dimensional $\mathbb{F}_{q_0}\langle x\rangle$-module. As every non-trivial irreducible $\mathbb{F}_{q_0}\langle x'\rangle$-submodule of $V$ has dimension $\ell$ and as $\ell\geq 2$,  it follows that $x'$ acts trivially on $w$ and $w\in \cent V {x'}\cap[V,x']=0$. This proves that
\begin{equation}\label{verbose}\textrm{if }\langle v\rangle^x=\langle v\rangle, \textrm{ then }v\in \cent V {x'}.
\end{equation}
Recall that $\dim_{\mathbb{F}_{q_0}}\cent V{x'}=n-\ell'$.

 First suppose that $G$ is linear so that \textbf{Case}~(i) holds and $\Omega$ is the set of totally singular $1$-subspaces. From~\eqref{verbose}, we have $\fpr_\Omega(x)= (q_0^{n-\ell'}-1)/(q_0^n-1)\leq q_0^{-\ell'}$. 

Suppose that $G$ is unitary. Then, the Hermitian form on $V$ restricts to a non-degenerate Hermitian form on $\cent V{x'}$. Therefore, in \textbf{Case}~(i)  from~\eqref{verbose} we get that the number of totally singular $1$-subspaces fixed by $x$ is $(q^{n-\ell'}-(-1)^{n-\ell'})(q^{n-\ell'-1}-(-1)^{n-\ell'-1})/(q^2-1)$ and hence 
\[\mathrm{fpr}_\Omega(x)=\frac{(q^{n-\ell'}-(-1)^{n-\ell'})(q^{n-\ell'-1}-(-1)^{n-\ell'-1})}{(q^n-(-1)^n)(q^{n-1}-(-1)^{n-1})}\leq \frac{2}{q^{2\ell'}}=\frac{2}{q_0^{\ell'}}\]
(the inequality follows from a direct computation distinguishing the parity of $n$ and $\ell'$). 
Similarly, in \textbf{Case}~(ii) from~\eqref{verbose} we get that the number of non-degenerate $1$-subspaces fixed by $x$ is $(q^{n-\ell'}-(-1)^{n-\ell'})q^{n-\ell'-1}/(q+1)$ and hence
\[\fpr_\Omega(x)=\frac{(q^{n-\ell'}-(-1)^{n-\ell'})q^{n-\ell'-1}}{(q^n-(-1)^n)q^{n-1}}\leq \frac{2}{q^{2\ell'}}=\frac{2}{q_0^{\ell'}}.\]

Now suppose that $G$ is orthogonal. From Remark~\ref{remarkorbits}, in \textbf{Case }(i) the set $\Omega$ consists of all singular $1$-subspaces of $V$, and hence 
\[
|\Omega|=
\begin{cases}
(q^{n-1}-1)/(q-1)&\textrm{if }G_0=\POmega_{n}(q),\\
(q^{n/2}-1)(q^{n/2-1}+1)/(q-1)&\textrm{if }G_0=\POmega_{n}^+(q),\\
(q^{n/2-1}-1)(q^{n/2}+1)/(q-1)&\textrm{if }G_0=\POmega_{n}^-(q).\\
\end{cases}
\]
We denote by $\Omega_n^\circ$, $\Omega_n^+$ and $\Omega_n^-$ these three  $G$-sets, respectively. Note that for every $n$ and $q$ we have 
\begin{equation}\label{eq42}
|\Omega_n^-|\leq |\Omega_n^\circ|\leq |\Omega_n^+|.
\end{equation}
The quadratic form $Q$ on $V$ restricts to a non-degenerate quadratic form $Q'$ on $\cent V{x'}$. Observe that the type of $Q'$ depends on both $Q$ and $\cent V {x'}$. (In particular, if $q$ is even then $n-\ell'$ is even.) A maximal totally isotropic subspace of $\cent V{x'}$ has dimension $(n-\ell'-1)/2$ (if $n-\ell'$ is odd), $(n-\ell')/2$ (if $n-\ell'$ is even and $Q'$ is of ``plus type''), and $(n-\ell')/2-1$ (if $n-\ell'$ is even and $Q'$ is of ``minus type''). Therefore
the number of totally singular $1$-subspaces fixed by $x$ is either $(q^{n-\ell'-1}-1)/(q-1)$, or $(q^{(n-\ell')/2}-1)(q^{(n-\ell')/2-1}+1)$, or $(q^{(n-\ell')/2-1}-1)(q^{(n-\ell')/2}+1)/(q-1)$ in the three respective cases described above. Observe that when $\dim_{\mathbb{F}_q}\cent V{x'}\leq 1$ (that is, $\ell'\geq n-1$), the space $\cent V{x'}$ contains no singular points, and hence $\fpr_\Omega(x)=0$. For $\ell'\leq n-2$, using~\eqref{eq42}, with an easy case-by-case analysis studying the various possibilities for $Q$ and $Q'$ gives
\begin{eqnarray*}
\fpr_\Omega(x)&=&\frac{|\mathrm{Fix}_\Omega(x)|}{|\Omega|}\leq \frac{|\mathrm{Fix}_\Omega(x)|}{|\Omega_n^-|}\leq\frac{(q^{(n-\ell')/2}-1)(q^{(n-\ell')/2-1}+1)}{(q^{n/2-1}-1)(q^{n/2}+1)}\leq \frac{2}{q^{\ell'}}.
\end{eqnarray*}

Finally we consider {\bf Case }(ii). Assume that $q$ is even. In particular, $\Omega={\bf NS_1}$ and $|\Omega|=q^{n/2-1}(q^{n/2}-\varepsilon\cdot 1)$ when $G_0=\POmega_n^\varepsilon(q)$, $\varepsilon\in \{+,-\}$.  Observe that $\ell'$ is even.
Now the number of non-degenerate $1$-subspaces fixed by $x$ is either $q^{(n-\ell')/2-1}(q^{(n-\ell')/2}-1)$ (when $Q'$ has Witt index $(n-\ell')/2$) or $q^{(n-\ell')/2-1}(q^{(n-\ell')/2}+1)$ (when $Q'$ has Witt index $(n-\ell')/2-1$). For $\ell'=n$, we have $\fpr_\Omega(x)=0$, and for $\ell'<n$, we have 
\begin{equation*}
\mathrm{fpr}_\Omega(x)\leq \frac{|\mathrm{Fix}_\Omega(x)|}{|\Omega|}
\leq \frac{q^{(n-\ell')/2-1}(q^{(n-\ell')/2}+1)}{q^{n/2-1}(q^{n/2}-1)} 
\leq \frac{2}{q^{\ell'}}.
\end{equation*}
Finally, assume that $q$ is odd. By looking at the various possibilities for $|\Omega|$ in Remark~\ref{remarkorbits}, we get $|\Omega|\geq q^{(n-1)/2}(q^{(n-1)/2}-1)/2$. Now the number of non-degenerate $1$-subspaces fixed by $x$ is either $q^{n-\ell'-1}$ (when $n-\ell'$ is odd), or $q^{(n-\ell')/2-1}(q^{(n-\ell')/2}-1)$ (when $n-\ell'$ is even and $Q'$ has Witt index $(n-\ell')/2$) or $q^{(n-\ell')/2-1}(q^{(n-\ell')/2}+1)$ (when $n-\ell'$ is even and $Q'$ has Witt index $(n-\ell')/2-1$). If $\ell'=n$, then $\Fix_\Omega(x)=\emptyset$ and $\fpr_\Omega(x)=0$. If $\ell'=n-1$, then $\dim_{\mathbb{F}_q}\cent V {x'}=1$ and hence $\cent V {x'}$ is the only non-degenerate $1$-subspace fixed by $x$. Thus $\fpr_\Omega(x)=1/|\Omega|<36/(13q_0^{n-1})$ (where the inequality follows with a computation). Assume that $n-\ell'\geq 2$. Thus
\begin{equation*}
\mathrm{fpr}_\Omega(x)\leq \frac{|\mathrm{Fix}_\Omega(x)|}{|\Omega|}
\leq \frac{q^{(n-\ell')/2-1}(q^{(n-\ell')/2}+1)}{q^{(n-1)/2}(q^{(n-1)/2}-1)/2} =\frac{2}{q^{\ell'}}\frac{\left(1+\frac{1}{q^{(n-\ell')/2}}\right)}{\left(1-\frac{1}{q^{(n-1)/2}}\right)}
\leq \frac{36}{13q^{\ell'}},
\end{equation*}
where the last inequality follows by noticing that the maximum of the second factor in the above product is achieved for $q=3$, $n=7$ and $n-\ell'=2$.
 \end{proof}
We conclude this section by introducing a notation that will be used frequently in our arguments. Let $t$ be a prime power and let $\ell\geq 1$, then $$\omega_t(t^\ell-1)$$
denotes the number of primitive prime divisors of $t^\ell-1$, that is, the number of prime divisors $r$ of $t^\ell-1$ with $r\nmid t^i-1$ for every $1\leq i\leq\ell-1$.

\section{{\bf Case}~$(iii)$ of Proposition~\ref{reduction}: action on maximal totally singular subspaces}\label{Case3}

As usual we assume Notation~\ref{generalnotation}, and we start our analysis by considering {\bf Case}~$(iii)$ of Proposition~\ref{reduction}. We let $G$ be a primitive group with socle $G_0$ acting on the set $\Omega$ of totally singular $m$-subspaces of $V$. 
We need the following lemma.

\begin{lemma}[{\cite[Lemma~$3.14$]{GK}}]\label{importantiii} Suppose that $m\geq 3$.  Then, for $g\in G\setminus\{1\}$, we have
$\fpr_\Omega(g)<2/q_0^{m^*}+1/q_0^{m^\sharp}$, where $m^*$ and $m^\sharp$ are given in Table~$\ref{tablemm}$.
\end{lemma}

\begin{table}
\begin{tabular}{|l|l|l|l|l|l|l|}\hline
  $G_0$    & $\PSp_{2m}(q)$ & $\POmega_{2m}^+(q)$ & $\POmega_{2m+1}(q)$ & $\POmega^{-}_{2m+2}(q)$ & $\PSU_{2m}(q)$ & $\PSU_{2m+1}(q)$\\\hline
  $m^*$    &      $m$       &        $m-1$        &        $m$        &         $m+1$         &    $m-\frac{1}{2}$     &     $m+\frac{1}{2}$\\     
$m^\sharp$ &     $m-1$      &        $m-2$        &         $m-1$         &        $m$        &     $m-1$      &    $m$\\\hline    
\end{tabular}
\caption{Definitions of $m^{*}$ and $m^{\sharp}$}\label{tablemm}
\end{table}

\begin{proposition}\label{totiso}
If $g\in G$, then either $g$ has a regular cycle, or $(G_0,\Omega)\cong(\Alt(5),\{1,\ldots,5\})$.
\end{proposition}
\begin{proof}
Denote by $o$ the maximum order of an element of $G$. In particular, $|g|\leq o$. (Recall that form Remark~\ref{smallorth} we may assume that $n\geq 7$ when $G_0$ is orthogonal.)

Assume first that $m\geq 3$. (We use the notation introduced in Lemma~\ref{importantiii} and Table~\ref{tablemm}.) From Lemma~\ref{importantiii} we get 
\begin{eqnarray}
\label{5000}\quad \SGO 
&<&\sum_{\substack{r\textrm{ prime}\\r\mid |g|}}\left(\frac{2}{q_0^{m^*}}+\frac{1}{q_0^{m^\sharp}}\right)= \omega(|g|)\left(\frac{2}{q_0^{m^*}}+\frac{1}{q_0^{m^\sharp}}\right)\\ \nonumber
&\leq& \log_2(|g|)\left(\frac{2}{q_0^{m^*}}+\frac{1}{q_0^{m^\sharp}}\right)\leq
\log_2(o)\left(\frac{2}{q_0^{m^*}}+\frac{1}{q_0^{m^\sharp}}\right).
\end{eqnarray}
Now, an explicit upper bound for $o$ as a function of $q$ and $m$ can be found in~\cite[Table~$3$]{GMPS}. It follows  that $\log_2(o)(2/q_0^{m^*}+1/q_0^{m^\sharp})<1$ except when $G_0$ belongs to the following list: 
\[(\dag)\qquad\PSp_6(2),\PSp_8(2),\PSp_6(3),\mathrm{P}\Omega_8^+(2),\mathrm{P}\Omega_{10}^+(2),\mathrm{P}\Omega_{12}^+(2),\mathrm{P}\Omega_8^+(4),\mathrm{P}\Omega_8^+(3),\mathrm{P}\Omega_8^-(2),\mathrm{P}\Omega_7(3).\]
 In particular, apart from this handful of exceptions, the proof follows as usual from Lemma~\ref{lemma:apeman}. When $G_0 = \POmega_8^{+}(4)$, we replace $\log_2(o)$ by $6$ (the exact number of primes dividing $|G|$) in~\eqref{5000}, which shows that $\SGO<1$. The same argument works for $G_0= \POmega_{12}^{+}(2),\POmega_{8}^{+}(3) $, where we can replace $\log_2(o)$ by $7$ and $5$, respectively. For $G_0 = \POmega_{10}^{+}(2)$, we can calculate in \texttt{magma} all the possible element orders of $\Aut(G_0) = \mathrm{PGO}^{+}_{10}(2)$, and check that $\omega(|g|) \le 3$, for every $g \in G$. This refinement of~\eqref{5000} implies that $ \SGO < 1$. 
For each of the remaining groups $G$ arising from $G_0$ in $(\dag)$, we can construct in \texttt{magma} the permutation representation of $G$ on the maximal totally singular subspaces of $V$ and check directly that every conjugacy class representative has a regular cycle.

 Suppose that $m\leq 2$. In particular, $G_0$ is $\PSp_2(q)\cong \PSL_2(q)$, $\PSp_4(q)$, $\PSU_3(q)$, $\PSU_4(q)$, or $\PSU_5(q)$. Observe that by Theorem~\ref{smalldimension}, we may assume that $G_0=\PSU_5(q)$. By Theorem~\ref{LSthm}, we have  $\fpr_\Omega(x)\leq 4/(3q)$ for every $x\in G\setminus\{1\}$. In particular, the argument in the previous paragraph implies that
\begin{align} \label{5001}
\SGO = \sum_{\substack{r\textrm{ prime}\\r\mid |g|}}\mathrm{fpr}_\Omega(g^{|g|/r})\leq \frac{4\omega(|g|)}{3q}\leq \frac{4\log_2(o)}{3q}.
 \end{align}
Now an easy computation (using the value of $o$ in~\cite[Table~$3$]{GMPS}) shows that $\SGO < 1$ except when $q\in \{2,3,4,5,7,8,9,11,13,16,17,19,23\}$.
 Now observe that $\omega(|g|)\leq \omega(|\mathrm{Aut}(G_0)|)$ and that $\omega(|\mathrm{Aut}(G_0)|)$ can be explicitly computed for each remaining $G_0$. In particular, we see that $4\omega(|\mathrm{Aut}(G_0)|)/(3q)<1$ except when $q\in \{2,3,4,5,7,8\}$.
 We note that if $x$ has prime order $r$ and $r$ does not divide $|G_\alpha|$, for $\alpha\in\Omega$, then $\fpr_{\Omega}(x) =0$ (see Lemma~\ref{obvious3}). For $q=7$, the only primes dividing $|G_\alpha|$ are $2,3,5,7$ (see~\cite[Proposition~4.1.18]{KL} for example) so we can replace $\omega(|g|)$ by $4$ in~\eqref{5001} and we obtain $S(g,\Omega)<1$, the desired bound. Similarly, for $q=8$ we replace $\omega(|g|)$ by $5$, which eliminates this case. 
The remaining cases are $q=2,3,4,5$, and these are small enough to eliminate directly in \texttt{magma} by computing the permutation representation of $\PGammaU_5(q)$ on maximal totally singular subspaces. 
\end{proof}

\section{{\bf Case}~$(i)$ of Proposition~\ref{reduction}: action on totally singular $1$-subspaces}\label{Casei}
Here we let $G$ be a primitive group with socle $G_0$ acting on the set $\Omega$ of totally singular $1$-subspaces of $V$. 
The following lemma is in~\cite[Propositions~$3.1$~(ii) and~$3.15$]{GK}.

\begin{lemma}\label{importanti} Suppose that  $m\geq 3$ and $g\in G\setminus\{1\}$. Then 
\[
\mathrm{fpr}_\Omega(g)<\left\{
\begin{array}{lcl}
\min\{1/2,1/q+1/q^{n-1}\}&&\textrm{if }G_0=\PSL_n(q),\\
2/q_0^{m^*}+1/q^{m^\sharp}+1/q_0&&\textrm{otherwise},
\end{array}
\right.
\]
where $m^*$ and $m^\sharp$ are given in Table~$\ref{tablemm}$. 
\end{lemma}
\begin{proposition}\label{caseiiiPSLPSU}
If $g\in G$, then either $g$ has a regular cycle, or $(G_0,\Omega)\cong(\Alt(5),\{1,\ldots,5\})$.
\end{proposition}
\begin{proof}
When $m=1$, a maximal totally singular subspace of $V$ has dimension $1$ and hence, by Proposition~\ref{totiso}, we may assume that $m\geq 2$. 
Observe that when $G_0$ is $\PSL_n(q)$ or $\PSp_{n}(q)$ every $1$-subspace of $V$ is totally singular (recall that when $G_0=\PSL_n(q)$ the space $V$ is endowed with the trivial form). In particular, $\norm {{\Sym(\Omega)}}{\PSp_n(q)}\leq \norm {{\Sym(\Omega)}}{\PSL_n(q)}=\PGammaL_n(q)$. Therefore, in this proof, we may omit the symplectic groups from our analysis.

We now split the proof into various cases.

\smallskip

\noindent\textsc{Case $G_0=\PSL_n(q)$. } In this case $q_0=q$. Observe that when $q$ is prime, we have $G\leq \PGL_n(q)$ and hence the proof follows from~\cite[Theorem~$2.3$]{SZ}. In particular, we may assume that $e\geq 2$. 

Recall that if $x\in G\setminus\{1\}$, then $\fpr_\Omega(x)\leq 4/(3q)$.  Therefore $$\SGOp \le \omega(ep(q-1))\frac{4}{3q}.$$

From \Cref{fprell}, we see that if $x$ has order $r$, with $r\nmid ep(q-1)$, then $\fpr_\Omega(x)\leq q^{-\ell}$ where $\ell$ is the smallest positive integer with $r\mid q^\ell-1$. The number $\omega_q(q^\ell-1)$ of primitive prime divisors of $q^\ell-1$ is at most $\log_2(q^\ell-1)\leq \ell\log_2(q)$. Thus
\begin{equation*}
  \SGOpp \le \sum_{\ell=2}^\infty\frac{\omega_q(q^\ell-1)}{q^\ell} \le \sum_{\ell=2}^\infty\frac{\ell\log_2(q)}{q^\ell} = \log_2(q)\left(\frac{q}{(q-1)^2}-\frac{1}{q}\right)
 \end{equation*}
where the last equality follows from Lemma~\ref{series}. Observe that $ep\leq q$ and hence $\omega(ep(q-1))\leq \log_2(q^2)$. 

Using these upper bounds for $\SGOp$ and $\SGOpp$, it is easy to see that $\SGO< 1$ for $q\geq 11$. For $q \in \{8,9\}$, we can compute  $\omega(ep(q-1))$ explicitly from which it follows that $\SGO<1$. Finally, for $q=4$, we have $\omega(ep(q-1))=2$ and by Lemma~\ref{importanti} we have $\fpr_\Omega(x)<1/4+1/4^{n-1}\leq 1/4+1/4^4$ (recall that $n\geq 5$). Thus $\SGOp < 2\cdot (1/4+1/4^{4})$ and $\SGO<1$.

\smallskip

\noindent\textsc{Case $G_0=\PSU_n(q)$. }In this case $q_0=q^2$ and $|\Omega|=(q^n-(-1)^n)(q^{n-1}-(-1)^{n-1})/(q^2-1)$.

From \Cref{fprell}, we see that if $x$ has order $r$, with $r\nmid ep(q-1)$, then $\fpr_\Omega(x)\leq  2q_0^{-\ell}\leq 2q^{-2\ell}$ and, arguing as in the case $G_0=\PSL_n(q)$, we get 
\begin{equation} \label{JJJ}
  \SGOpp \le \sum_{\ell=2}^\infty\frac{2\omega_{q^2}(q^{2\ell}-1)}{q^{2\ell}} \le 4\log_2(q)\left(\frac{q^2}{(q^2-1)^2}-\frac{1}{q^2}\right).
 \end{equation}

Suppose that $n\geq 6$, that is, $m\geq 3$. Using Lemma~\ref{importanti}, we have $\SGOp \le \omega(ep(q^2-1))(2q^{-2m^*}+q^{-m^\sharp}+q^{-2})$. 
It follows that $\SGO < 1$ for $q\neq 2$. 
If $q=2$, then  $q^2-1=3$, $q^4-1=15=3\cdot 5$ and $q^6-1=63=3^2\cdot 7$. Hence $\omega_{q^2}(q^4-1)=1$ and $\omega_{q^2}(q^6-1)=1$, and so for the first two terms of the summation in~\eqref{JJJ} we can take $2/q^4$ and $2/q^6$. Therefore
\[
\SGOpp\le 4\left(\frac{q^2}{(q^2-1)^2}-\frac{1}{q^2}-\frac{2}{q^4}-\frac{3}{q^6}\right)+\frac{2}{q^4}+\frac{2}{q^6}.
\]
For $q=2$, it follows that  $\SGO< 1$ unless $n\in\{6,7,8\}$.
We verify the cases $G_0 = \PSU_6(2)$, $\PSU_7(2)$ and $\PSU_8(2)$ directly in \texttt{magma}.

It remains to consider $G_0=\PSU_5(q)$. Here we observe that $\SGOp \le \omega(ep(q^2-1))\frac{4}{3q}$ by Theorem~\ref{LSthm}.
Using~\eqref{JJJ}, it is easy to check that $\SGO < 1$ for $q\geq 5$. 
 The cases $q=2,3,4$  
 follow by direct computations in \texttt{magma}.

\smallskip

\noindent\textsc{Case $G_0$ is orthogonal. }From Remark~\ref{smallorth}, we have $n\geq 7$. Using Lemma~\ref{importanti} and arguing exactly as in the linear and unitary case, we have $\SGOp \le  \omega(ep(q-1))(2/q^{m^*}+1/q^{m^\sharp}+1/q)$ and 
\begin{equation} \label{JJJJ}
  \SGOpp \le \sum_{\ell=2}^\infty\frac{2\omega_{q}(q^{\ell}-1)}{q^{\ell}} \le 2\log_2(q)\left(\frac{q}{(q-1)^2}-\frac{1}{q}\right).
 \end{equation}
It follows that $\SGO< 1$ for $q \ge 7$.
If  $q=5$, then $\omega_q(q^2-1)=1$ and hence for the first term in the summation in~\eqref{JJJJ} we can take $2/q^2$. In particular, 

\[
\SGOpp\le2\log_2(q)\left(\frac{q}{(q-1)^2}-\frac{1}{q}-\frac{2}{q^2}\right)+\frac{2}{q^2}.
\]
With this upper bound on $\SGOpp$, it follows that $\SGO<1$.
For $q=4$, we have 
$\omega_q(q^2-1)=\omega_q(q^3-1)=1$ and the usual argument shows that   $\SGO< 1$.
Similarly, if  $q=3$, then $\omega_q(q^2-1)=0$ and $\omega_q(q^3-1)=\omega_q(q^4-1)=1$ and it follows that that $\SGO<1$ unless $G_0 = \POmega_7(3), \POmega_8^+(3)$. 
We eliminate the remaining two cases with the invaluable help of \texttt{magma}.
%

It remains to deal with the case $q=2$. In particular,  $n$ is even. 
If $n\geq 10$, then~\cite[Tables~$2.1$C and~D]{KL} shows that $\Aut(G_0)$ is the  isometry  group $I(V)$ of the orthogonal space $V$, and hence  $G \le I(V)$. When $n=8$, the triality automorphism of $\Omega_8^+(2)$ does not preserve the action of $G_0$ on the totally singular subspaces, and hence $G\leq I(V)$ in this case as well. Therefore the proof follows from~\cite[Theorem~$1.2$]{EZ}.
\end{proof}

\section{{\bf Case}~$(ii)$ of Proposition~\ref{reduction}: action on non-degenerate $1$-subspaces}\label{caseii}

Here we let $G$ be a primitive group with socle $G_0$ acting on a set $\Omega$ consisting of non-degenerate $1$-subspaces  of $V$. The various possibilities for $\Omega$ (depending on the type of $G_0$) are discussed in Remark~\ref{remarkorbits}.


\begin{proposition}\label{NS1}If $g\in G$, then $g$ has a regular orbit on $\Omega$.
\end{proposition}
\begin{proof}
We subdivide the proof depending on the type of $G_0$. Here $G_0$ is either unitary or orthogonal.

\smallskip

\noindent\textsc{Case $G_0=\PSU_n(q)$. }
We may assume that $ n\ge 5$ by Theorem~\ref{smalldimension}.
 First suppose that $n=5$. Let $\omega$ be the number of prime divisors of $|\Aut(G_0)|$. Clearly, it suffices to show that $4\omega /(3q)<1$. Now $\omega\leq \log_2(|\Aut(G_0)|)$ and we see that the inequality $4\log_2(|\Aut(G_0)|)/(3q)<1$ is satisfied for $q\geq 261$. For  $q\leq 260$, we can compute the exact value of $\omega$ and we see that $4\omega /(3q)<1$ when $q\geq 9$. For $q\in  \{7,8\}$, replacing $\omega$ with $\omega'=\max(\omega(|h|)\mid h\in \Aut(G_0))$,  we obtain  $4\omega'/(3q)<1$. For $q=5$, we check in \texttt{magma} that the number of prime divisors of $|g|$ is at most $3$, unless $|g|=210,630$, in which case the prime divisors of $|g|$ are $2,3,5,7$. By Proposition~\ref{fprell}, for $r=7$ we have $\fpr_{\Omega}(g^{|g|/7}) < 2/25^{3}$ and with this refinement, $\SGO  \le 3 \cdot 4/(3q)  +2/25^{3}<1$.
For $q\in \{2,3,4\}$, we construct the explicit permutation representation of the action of $\mathrm{P}\Gamma \mathrm{U}_5(q)$ on $\Omega=\mathbf{NS}_1$ and check that every element has a regular orbit. 
 
We now suppose that $n\geq 6$ so that $m\geq 3$. For every element $x\in G\setminus\{1\}$, we see from~\cite[Proposition~$3.16$]{GK} that 
\begin{equation} \label{fprNS1}
\mathrm{fpr}_\Omega(x)< f(n,q):=
\begin{cases}
2/q^{2(m-2)}+1/q^{2m-1}+1/q^{2(m-1)}+1/q^2&\textrm{if }n=2m \textrm{ and }m\geq 3,\\
1/q^{2m+1}+1/q^{2m+1}+1/q^{2m}+1/q^2&\textrm{if }n=2m+1 \textrm{ and }m\geq 3.
\end{cases}
  \end{equation}

By \eqref{fprNS1}, we have $\SGOp \le \omega(ep(q^2-1))f(n,q)$ and, from Proposition~\ref{fprell}, we have 
\begin{equation}\label{JJJj}
\SGOpp\leq \sum_{\ell=2}^\infty\frac{2\omega_{q^2}(q^{2\ell}-1)}{q^{2\ell}} \le \sum_{\ell=2}^\infty\frac{2\ell\log_2(q^2)}{q^{2\ell}} =4\log_2(q)\left(\frac{q^2}{(q^2-1)^2}-\frac{1}{q^2}\right).
\end{equation} 
It follows that $\SGO<1$ for $q\neq 2$. 
(To see this, for $q\geq 17$, observe that $ep(q^2-1)\leq q^3$ and hence $\omega(ep(q^2-1))\leq \log_2(q^3)$. For $q< 17$, use the explicit value of $\omega(ep(q^2-1))$.) Now suppose that $q=2$. Then $q^2-1=3$, $q^4-1=3\cdot 5$ and $q^6-1=3^2\cdot 7$. Hence $\omega_{q^2}(q^4-1)=1$ and $\omega_{q^2}(q^6-1)=1$, and so for the first two terms of the summation in~\eqref{JJJj} we can take $2/q^4$ and $2/q^6$. Thus
\[
\SGOpp\leq 4\left(\frac{q^2}{(q^2-1)^2}-\frac{1}{q^2}-\frac{2}{q^4}-\frac{3}{q^6}\right)+\frac{2}{q^4}+\frac{2}{q^6},
\]
and now $\SGO< 1$ for every $n\geq 9$. 
 Finally for the groups $G$ with socle $G_0 =\mathrm{PSU}_n(2)$ and $n=6,7,8$,
the proof follows by calculating $\SGO$ precisely in \texttt{magma}.

\medskip

\noindent\textsc{Case $G_0$ is orthogonal. }
From Remark~\ref{smallorth}, $n\geq 7$. For every $x\in G\setminus\{1\}$,~\cite[Proposition~$3.16$]{GK} (applied with $k=n-1$) gives 
\begin{equation} \label{fepsnq}
\mathrm{fpr}_\Omega(x)< f(n,q)\quad\textrm{where}\quad f(n,q):=\frac{2}{q^{m^*}}+\frac{2}{q^{m^\sharp}}+\frac{1}{q}
\end{equation}
(where $m^\sharp$ and $m^*$ are as in Table~\ref{tablemm}).
Using \eqref{fepsnq} and \Cref{fprell} and arguing as in the unitary case we find that $\SGOp \le \omega(ep(q-1))f(n,q)$ and 
\begin{equation} \label{JJJJjj}
 \SGOpp \le  \sum_{\ell=2}^\infty\frac{36\omega_{q}(q^{\ell}-1)}{13q^{\ell}}
\le\frac{36}{13}\log_2(q)\left(\frac{q}{(q-1)^2}-\frac{1}{q}\right).
 \end{equation}
It follows that $\SGO < 1$ for $q \ge 7$, unless $G_0 \in\{\POmega_7(7), \POmega_{8}^{+}(7)\}$. 
In the case $G_0\in \{ \POmega_7(7),\POmega_{8}^{+}(7)\}$, we note that the only primes dividing $|G|$ but not $ep(q-1)$ are $5,19,43$. Thus $ \SGOpp \le 3 \cdot  \frac{36}{13q^{2}}$ and $\SGO < 1$.


Now suppose that $q=5$. We have $\omega(ep(q-1))=2$ and $\omega_q(q^2-1)=1$, so $\SGOp \le 2 \cdot 4/(3q) = 8/15$
and 
\begin{equation*} 
 \SGOpp\le \frac{36}{13}\log_2(q)\left(\frac{q}{(q-1)^2}-\frac{1}{q}-\frac{2}{q^2}\right)+\frac{36}{13q^2}<\frac{7}{15}.
 \end{equation*}
Thus $\SGO<1$.

If $q=4$, then $\omega(ep(q-1))=2$ and $\omega_q(q^2-1)= \omega_q(q^3-1)=\omega_q(q^{4}-1)=1$, so $\SGOp\leq 2\cdot 4/(3q)=2/3$ and 
\[ \SGOpp \le \frac{36}{13}\log_2(q)\left(\frac{q}{(q-1)^2}-\frac{1}{q} - \frac{2}{q^2} - \frac{3}{q^3} - \frac{4}{q^4} \right)+\frac{36}{13q^2}+\frac{36}{13q^3}+\frac{36}{13q^4} <\frac{1}{3}.\]

For $q=3$, we have $\omega(ep(q-1))=2$ and $\SGOp\leq 2\cdot f(n,q)$. Moreover, $\omega_q(q^2-1)=0$ and $\omega_q(q^i-1)=1$, for $i = 3,4,5,6$. The usual argument shows that $\SGO <1$, unless $G_0\in \{\POmega_8^+(3),\POmega_{10}^+(3),\POmega_7(3),\POmega_9(3),\POmega_8^-(3)\}$. For these remaining groups, the proof follows by calculating $\SGO$ precisely in \texttt{magma}.

Finally, suppose that $q=2$, and hence $n$ is even. If $n\geq 10$, then~\cite[Tables~2.1C and D]{KL} shows that $\Aut(G_0)$ is the isometry group $I(V)$ of the orthogonal space $V$, and hence $G\leq I(V)$. When $n=8$, the triality automorphism of $\Omega_8^+(2)$ does not preserve the action of $G_0$ on $\Omega={\bf NS}_1$, and hence $G\leq I(V)$ in this case as well. Therefore the proof follows from~\cite[Theorem~$1.2$]{EZ}.
\end{proof}

\section{{\bf Case}~$(iv)$ of Proposition~\ref{reduction}: action on non-degenerate $2$-subspaces with no singular point}\label{Caseiv}

Here we let $G$ be a primitive group with socle $G_0$ acting on the set $\Omega$ consisting of all non-degenerate $2$-subspaces of $V$ that contain no singular point. This means that the quadratic form on $V$ restricted to each of these $2$-dimensional subspaces is anisotropic~(see~\cite[Section~$6.3$]{Cameron} for example). In particular, $G_0$ is an orthogonal group. Moreover, in the light of Remark~\ref{smallorth}, $G_0$ is either $\POmega_{2m+1}(q)$ (with $m\geq 3$), or $\POmega_{2m}^+(q)$ (with $m\geq 4$) or $\POmega_{2m+2}^-(q)$ (with $m\geq 3$). 
We recall that 
\[
|\Omega|=
\begin{cases}
 \frac{1}{2}q^{2(m-1)}(q^m-1)(q^{m-1}-1)(q+1)^{-1}&\textrm{if }G_0=\POmega_{2m}^+(q),\\
\frac{1}{2}q^{2m}(q^{m+1}+1)(q^{m}+1)(q+1)^{-1}&\textrm{if }G_0=\POmega_{2m+2}^-(q),\\
\frac{1}{2}q^{2m-1}(q^{2m}-1)(q+1)^{-1}&\textrm{if }G_0=\POmega_{2m+1}(q).
\end{cases}
\]
(these formulae can be deduced from the index of $G_0\cap (\mathrm{O}_2^-(q)\perp \mathrm{O}_{n-2}^\varepsilon(q))$ in $G_0$). When necessary we write $\Omega_n^+$, $\Omega_n^-$ and $\Omega_n^\circ$ for $\Omega$, according to the corresponding type of $G_0$. By writing $|\Omega|$ as a function of $n$ and $q$ (so that $m=n/2$, $n/2-1$ or $(n-1)/2$ respectively), we get
\begin{equation}\label{symplify}
|\Omega_n^-|>|\Omega_n^\circ|>|\Omega_n^+|.
\end{equation}
\begin{lemma}\label{new2ani}
For $x\in G\setminus\{1\}$, we have 
\[\mathrm{fpr}_\Omega(x)\leq f(n,q)\quad\textrm{where}\quad f(n,q):=3/q^{n/2-2}+1/q^{n/2-1}+1/q^2.\]
\end{lemma}
\begin{proof}
This follows from~\cite[Proposition~$3.16$]{GK} (and the comment following its proof) applied with $k=d-2$ and by writing $m$ as a function of $n$ in each case. It is straightforward to see that the worst upper bound for $\fpr_\Omega(x)$ arises when $n=2m$ and $G_0=\POmega_{2m}^+(q)$. 
\end{proof}
\begin{proposition}\label{prop}
If $g\in G$, then $g$ has a regular orbit on $\Omega$.
\end{proposition}
\begin{proof}
 Let $x\in G$ of prime order $r$ and suppose that $r$ is coprime to $ep(q^2-1)$; in particular $x\in G_0$ is semisimple.
Let $\ell$ be the minimum positive integer such that $r\mid (q^\ell-1)$. In particular we have $\ell\geq 3$. Let $x'$ be a semisimple element of $\GL(V)$ projecting to $x$. As in Proposition~\ref{fprell} we observe that  $V=[V,x']\oplus \cent V {x'}$ and  $\ell' =\dim_{\mathbb{F}_{q}}[V,x']\geq \ell$. 

Let $U\in \Omega$ with $U^x=U$. Since $\dim U=2$ and  $U$ is $x'$-invariant, we see that $x'$ acts trivially on $U$ and $U \leq \cent V {x'}$. This proves that
\begin{equation}\label{verbose2}\textrm{if }U^x=U, \textrm{ then }U\leq \cent V {x'}.
\end{equation}

The quadratic form $Q$ on $V$ restricts to a non-degenerate quadratic form $Q'$ on $\cent V {x'}$. Therefore from~\eqref{verbose2} we obtain that the number of elements of $\Omega$ fixed by $x$ is $|\Omega_{n-\ell'}^\varepsilon|$, where $\varepsilon\in\{+,-,\circ\}$.
In particular,~\eqref{symplify} implies that for fixed  $\ell'$, the largest fixed-point-ratio occurs when $G_0=\POmega_{2m}^+(q)$ and $Q'$ is of $-$ type (that is, when $Q'$ has Witt index $(n-\ell')/2-1$). If $\ell'\geq n-1$, then $\dim_{\mathbb{F}_{q}}\cent V {x'}\leq 1$ and $\fpr_\Omega(x)=0$, and if $\ell'\leq n-2$, it follows that
\[\mathrm{fpr}_\Omega(x)\leq \frac{q^{n-\ell'-2}(q^{(n-\ell')/2}+1)(q^{(n-\ell')/2-1}+1)}{q^{n-2}(q^{n/2}-1)(q^{n/2-1}-1)}\leq \frac{4}{q^{2\ell'}}.\]

 With an argument similar to the cases that we discussed so far, we have
\begin{eqnarray}\label{J-J}\nonumber
\SGO&=&
\sum_{\substack{r\mid |g|\\ r\mid ep(q^2-1)}}
\mathrm{fpr}_\Omega(g^{|g|/r})+
\sum_{\substack{r\mid |g|\\ r\nmid ep(q^2-1)}}
\mathrm{fpr}_\Omega(g^{|g|/r})\\\nonumber
&\leq& \omega(ep(q^2-1))f(n,q)+\sum_{\ell=3}^\infty\frac{4\omega_q(q^\ell-1)}{q^{2\ell}}\\\nonumber
&\leq&\omega(ep(q^2-1))f(n,q)
+4\log_2(q)\left(\frac{q^2}{(q^2-1)^2}-\frac{1}{q^2}-\frac{2}{q^4}\right),
\end{eqnarray} 
where $f(n,q)$ is as in Lemma~\ref{new2ani}.

Since $\omega(ep(q^2-1))\leq \log_2(q^3)$ and $n\geq 7$, we have $\SGO< 1$ for $q>9$. For $q\leq 9$, we compute the exact value of $\omega(ep(q^2-1))$ and we see that $\SGO<1$, unless $q=2$ or $(n,q)=(7,3)$. The case $(n,q)=(7,3)$ can be easily eliminated using \texttt{magma}. Finally if $q=2$, then $G$ is contained in the isometry group of $ I(V)$ and hence we can apply~\cite[Theorem~$1.2$]{EZ}.
\end{proof}

\section{{\bf Case}~$(vi)$ of Proposition~\ref{reduction}: exceptional subspace actions of $\Sp_{n}(2^{e})'$}\label{casevi}

In this section $G_0=\Sp_{n}(q)'$ with $q=2^{e}$, and we let $G$ be a primitive permutation group with socle $G_0$ such that the stabilizer in $G_0$ of a point is isomorphic to either $\mathrm{O}_{n}^+(q)$ or $\mathrm{O}_{n}^-(q)$. We denote these two $G$-sets by $\Omega^+$ and $\Omega^-$, respectively. From \Cref{smalldimension}, we may assume that $n\geq 6$. In particular, $G_0=\Sp_{n}(q)$ and $m=n/2\geq 3$. We have the following useful lemma. 
\begin{lemma}\label{first}
Suppose that $m \ge 3$ and $q=2^e$. Let $x\in \Sp_{2m}(q)\setminus\{1\}$ be semisimple and set $c=\dim_{\mathbb{F}_q}\cent V x$. If $c>0$ or $\varepsilon=+$, then $\fpr_{\Omega^\varepsilon}(x)\leq 4q^{-2m+c}$, and if $c=0$ and $\varepsilon =- $, then $\fpr_{\Omega^\varepsilon}(x)\leq 4/(q^{m}(q^m-1))$.
\end{lemma}
\begin{proof} 
Replacing $x$ by a suitable power, we may assume that $x$ has prime order $r$. Let $\ell$ be the smallest positive integer such that $r \mid q^\ell-1$. Let $1, \omega_1,\ldots, \omega_{r-1}$ be the $r$th roots of unity in $\mathbb{F}_{q^\ell}$ and note that the Frobenius automorphism $\tau:\omega\mapsto \omega^q$ of $\mathbb{F}_{q^\ell}$ induces a permutation of the set $\{1,\omega_1,\ldots,\omega_{r-1}\}$. Let 
 $\{1\}, \Delta_1, \ldots, \Delta_t$ be the orbits of $\langle\tau\rangle$ on $\{1,\omega_1,\ldots,\omega_{r-1}\}$. Note that the multiset of eigenvalues of $x$ is a (multiset) union of these orbits. Observe that $c$ denotes the multiplicity of $\{1\}$  since $c= \dim_{\mathbb{F}_q} \cent V x$, and let $a_i$ denote the multiplicity of $\Delta_i$ in the multiset of eigenvalues of $x$. Thus the multiset of eigenvalues of $x$ can be parametrized by the $(t+1)$-tuple $(c,a_1, \ldots, a_t)$. 
 
Let $H$ denote the stabilizer in $G$ of a point of $\Omega^\varepsilon$. We claim that 
\begin{equation}\label{delta}
|x^G|=|x^{\Sp_n(q)}|\quad\textrm{and}\quad|x^G\cap H|\leq 2^{\delta_{c,0}}|x^{\O_n^\varepsilon(q)}|,
\end{equation}
where the Kronecker delta $\delta_{c,0}=1$ if $c=0$ and $\delta_{c,0}=0$ otherwise. Observe that to prove the inequality in~\eqref{delta} we need to show that $x^G\cap H$ is an $\O_n^\varepsilon(q)$-conjugacy class (when $c \ge 1$) and is the union of at most two $\O_n^{\varepsilon}(q)$-conjugacy classes (when $c=0$). 
Let $\mathbb{F}$ be the algebraic closure of $\mathbb{F}_q$, let $X$ denote the algebraic group $\Sp_{2m}(\mathbb{F})$ and $Y$ the algebraic group $\mathrm{PSO}_{2m}(\mathbb{F})$. Note that~\cite[Lemma~3.34]{Tim2} implies that $\cent X x $ and $\cent Y x$ are both connected. By~\cite[Corollary~3.7]{Tim2}, we see that two elements $g$ and $h$ of $Y_{\sigma} = \mathrm{Inndiag}(\mathrm{P}\Omega_{2m}^{\varepsilon}(q))$ are $Y_{\sigma}$-conjugate if and only if they are $Y$-conjugate and, in turn, by~\cite[4.2.2(j)]{GLS}, this is true if and only if $g$ and $h$ are $\mathrm{P}\Omega_{2m}^{\varepsilon}(q)$-conjugate. Furthermore, it follows from~\cite[Lemma~3.39]{Tim2} that if $g$ and $h$ in $Y$ have the same multiset of eigenvalues and if $1$ is an eigenvalue of $g$, then $g$ and $h$ are $Y$-conjugate and hence $\mathrm{P}\Omega_{2m}^{\varepsilon}(q)$-conjugate. Again from~\cite[Lemma~3.39]{Tim2}, if $g$ and $h$ in $Y$ have the same multiset of eigenvalues and if $1$ is not an eigenvalue of $g$, then there are at most two $Y$-classes to which $g$ and $h$ can belong. This proves~\eqref{delta}.

Combining Lemma~\ref{obvious3} and~\eqref{delta}, we have
 \begin{equation} \label{1eq}
 \mathrm{fpr}_{\Omega^\varepsilon}(x)
=\frac{|x^G\cap H|}{|x^G|} 
\leq\frac{2^{\delta_{c,0}}|x^{\O_{2m}^\varepsilon(q)}|}{|x^{\Sp_{2m}(q)}|}
 = \frac{2^{\delta_{c,0}} |\mathrm{O}_{2m}^{\varepsilon}(q)|}{|\Sp_{2m}(q)|} \!\cdot \! \frac{ |\cent {\Sp_{2m}(q)} x|}{ |\cent {\mathrm{O}_{2m}^{\varepsilon}(q)} x|}.
 \end{equation}
Explicit formulae for the order of the centralizers can be found in~\cite[Table~3.6]{Tim2}. If $\ell$ is odd, then 
\begin{align*} 
  \frac{ |\cent {\Sp_{2m}(q)} x|}{ |\cent {\mathrm{O}_{2m}^{\varepsilon}(q)} x|} = \frac{|\Sp_c(q)|}{2^{\delta_{c,0}-1} |\mathrm{O}_{c}^{\varepsilon}(q)|} &\quad\mbox{ and } &  \frac{|\mathrm{O}_{2m}^{\varepsilon}(q)|}{|\Sp_{2m}(q)|} = \frac{2q^{m^2-m}}{q^{m^2} (q^m+\varepsilon\cdot 1)} = \frac{2q^{-m}}{(q^{m}+\varepsilon\cdot 1)}.
\end{align*}
 Therefore, for $c>0$, the right hand side of~\eqref{1eq} becomes 
 \begin{align*} 
 \frac{2 q^{-m}}{(q^{m}+ \varepsilon\cdot 1)} \cdot \frac{(q^{c/2}+ \varepsilon\cdot 1) }{ q^{-c/2}} \le 4 q^{-2m+c} 
 \end{align*}
 (for the inequality observe that $(q^{c/2}+\varepsilon\cdot 1)/(q^{m}+\varepsilon\cdot 1)\leq (q^{c/2}+1)/(q^{m}-1)\leq 2q^{c/2-m}$).
 When $\ell$ is even we have very similar calculations: the only difference is that 
 \[ \frac{ |\cent {\Sp_{2m}(q)} x|}{ |\cent {\mathrm{O}_{2m}^{\varepsilon}(q)} x|} = \frac{|\Sp_c(q)|}{2^{\delta_{c,0}-1} |\mathrm{O}_{c}^{\varepsilon'}(q)|},\]
 where $\varepsilon'$ can be either of $+$ or $-$. For $c>0$, the right hand side of~\eqref{1eq} then becomes
 \begin{align*} 
 \frac{2 q^{-m}}{(q^{m}+ \varepsilon\cdot 1)} \cdot \frac{(q^{c/2}+ \varepsilon'\cdot 1) }{ q^{-c/2}} \le \frac{2 q^{-m}}{(q^{m}-1)} \cdot \frac{(q^{c/2}+ 1) }{ q^{-c/2}} \leq 4 q^{-2m+c} 
 \end{align*} 
 as before.

For $c=0$,~\eqref{1eq} gives $\fpr_{\Omega^\varepsilon}(x)\leq 2|\O_{2m}^\varepsilon(q)|/|\Sp_{2m}(q)|=4q^{-m}/(q^m+\varepsilon\cdot 1)$, and the lemma follows.
\end{proof}

\begin{proposition}\label{propcasevi}For $\varepsilon\in \{+,-\}$, if $g\in G$, then $g$ has a regular orbit on $\Omega^\varepsilon$.
\end{proposition}
\begin{proof}
 If $q=2$, then $G=G_0$ is simple and hence the proof follows from~\cite{EZ}. So we assume that $q\geq 4$. The usual computations yield $S_1(g,\Omega^\varepsilon) \le  \omega(2e(q-1)) \frac{4}{3q}$. Let $x \in G$ of prime order $r$ and suppose that $r$ is coprime to $2e(q-1)$, so that, in particular, we have $x\in G_0$.

Let $\ell$ be the minimum positive integer with $r\mid (q^\ell-1)$. In particular $\ell\geq 2$. Since $r$ is odd, we see that $x$ is a semisimple, and arguing as usual we see that $\cent V x$ has dimension at most $2m-\ell$. It follows from Lemma~\ref{first} that
$\mathrm{fpr}_{\Omega^\varepsilon}(x)\leq 4q^{-\ell}$ when $\ell<2m$, and $\mathrm{fpr}_{\Omega^\varepsilon}(x)\leq 4/(q^m(q^{m}-1))=4q^{-2m}+4(q^{2m}(q^m-1))^{-1}$ when $\ell=2m$.
The usual computations yield  
\begin{eqnarray}\label{eq28}\nonumber
 S_2(g,\Omega^\varepsilon) &\le&  4\sum_{\ell=2}^{2m}\frac{\omega_q(q^\ell-1)}{q^\ell}+\frac{4\omega_q(q^{2m}-1)}{q^{2m}(q^m-1)}\\
&\leq& 4\sum_{\ell=2}^\infty\frac{\omega_q(q^\ell-1)}{q^{\ell}} +\frac{4\omega_q(q^{2m}-1)}{q^{2m}(q^m-1)}\le 4\log_2(q)\left(\frac{q}{(q-1)^2}-\frac{1}{q}\right)+\frac{4\omega_q(q^{2m}-1)}{q^{2m}(q^m-1)}.
 \end{eqnarray}
Since $\omega(2e(q-1))\leq \log_2(q^2)=2e$ we see that $S(g,\Omega^\varepsilon) < 1$ for $e \ge 4$. If $e=3$, then  $q=8$, and $\omega(2e(q-1))=3$, which is enough to show that 
$S(g,\Omega^\varepsilon)<1$. Finally, suppose that $e=2$. Now, the only primes dividing $e p (q-1)$ are $2$ and $3$.  If  $x\in G$ has order $3$, then $\dim_{\mathbb{F}_q}\cent V x\ge 2m-2$ and  hence
$\fpr_{\Omega^\varepsilon}(x)\leq 4/q^2$ by Lemma~\ref{first}; thus $S_1(g,\Omega^\varepsilon) \le 4/(3q) + 4/q^{2}=7/12$. Also $\omega_q(q^2-1)=\omega_q(q^3-1)=\omega_q(q^4-1)=1$ so for the  first three terms of the summation in~\eqref{eq28}, we can take $4/q^{2}, 4/q^{3}$  and $4/q^{4}$. This shows that $\SGO<1$ since   we have 
\[ \SGOpp \le  \frac{4}{q^{2}} + \frac{4}{q^{3}} +\frac{4}{q^{4}} + 8\left(\frac{q}{(q-1)^2}-\frac{1}{q}-\frac{2}{q^2}-\frac{3}{q^3}-\frac{4}{q^4}\right)+\frac{4\omega_q(q^{2m}-1)}{q^{2m}(q^m-1)}< \frac{5}{12}.\]
(For the last inequality observe that $\omega_q(q^{2m}-1)\leq \log_2(q^{2m})=4m$ and that $4m/(q^{2m}(q^m-1))$ achieves its maximum at $m=3$.)
%
\end{proof}

Our proof of Proposition~\ref{propcasevi} is in line with the main techniques used in this paper, and is very much different from the arguments used in~\cite{SZ1} to deal with the case $q=2$. With some effort, an entirely different proof (in the same spirit as~\cite{SZ1}) can be obtained by following the arguments in~\cite[Section~$3$]{SZ1} and replacing~\cite[Proposition~$3.1$]{SZ1} with~\cite[Theorem~$1$]{GuPrSp}.

\section{{\bf Case}~$(v)$ of Proposition~\ref{reduction}: $G_0=\PSL_n(q)$, $G$ contains a graph automorphism and a point stabilizer is of type $P_{k,n-k}$ or $\GL_k(q) \oplus \GL_{n-k}(q)$}\label{casev}
In this section, we let $G$ be a primitive group on $\Omega$ with socle $G_0=\PSL_n(q)$ and containing a graph automorphism. For the time being, we are interested in only two subspace actions of $G$. We suppose that, for some $k$ with $1\leq k<n/2$, we have 
\[\Omega=\Omega_{k,\leq}\quad\textrm{where}\quad\Omega_{k,\leq}:=\{\{W,U\}\mid W,U\leq V,\,\dim W=k,\,\dim U=n-k,\,W\leq U\}\]
or
\[
\Omega=\Omega_{k,\perp}\quad\textrm{where}\quad\Omega_{k,\perp}:=\{\{W,U\}\mid W,U\leq V,\,\dim W=k,\,\dim U=n-k,\,V=W\oplus U\}.\\
\]
 In both cases, $\PGammaL_n(q)$ acts transitively on $\Omega$ and, using the notation in~\cite{KL}, the stabilizer of a point of $\Omega$ is a subspace subgroup of type $P_{k,n-k}$ (when $\Omega=\Omega_{k,\leq}$) or $\GL_k(q)\oplus\GL_{n-k}(q)$ (when $\Omega=\Omega_{k,\perp}$).

We endow $V$ with a non-degenerate bilinear form and we denote by $\tau$ the involutory automorphism of the projective geometry mapping $W$ to $W^\perp$, for $W\leq V$. Observe that $\tau$ normalizes $\PGammaL_n(q)$ and acts on $G_0$ as a graph automorphism. 

\begin{proposition}\label{propPSLgraph}
If $g\in G$, then $g$ has a regular cycle.
\end{proposition}
\begin{proof}
Suppose that $g\in G\cap \PGammaL_n(q)$. By Propositions~\ref{reduction},~\ref{totiso} and~\ref{caseiiiPSLPSU}, there exists a $k$-subspace $U$ of $V$ with $\langle g\rangle$ inducing in its action on $k$-subspaces a regular orbit on $\{U^h\mid h\in \langle g\rangle\}$. Let $\omega=\{U,W\} \in \Omega$, for some $(n-k)$-subspace $W$ of $V$. Clearly, $\langle g\rangle$ induces a regular orbit on $\omega$.

Suppose that $g\notin \PGammaL_n(q)$. In particular, $g$ has even order. From Lemma~\ref{basic}, we may assume that $g$ has square-free order. Write $g=xy=yx$, with $|x|=2$ and $|y|$ odd, that is,  $y \in \PGammaL_n(q)$. As $g\notin\PGammaL_n(q)$, we see that $x=\tau z$, for some $z\in \PGammaL_n(q)$. In particular, $x$ maps $k$-subspaces to $(n-k)$-subspaces. Let $W$ be a $k$-subspace with $y$ inducing a regular orbit on $\{W^h\mid h\in \langle y\rangle\}$. Note that since $y$ has odd order, we have $\{y^{1+2t}\mid t\geq 0\}=\langle y\rangle$.

Assume that there exists $U\notin \{W^{xy^{i}}\mid i\geq 0\}$ with $\omega=\{W,U\}\in \Omega$. If $\omega^{g^i}=\omega$ with $1\leq i<|g|$, then $i$ is odd because $g^2=y^2$ acts regularly on $\{W^{h}\mid h\in\langle y\rangle\}$. Therefore, $i=1+2t$, for some $t\geq 0$. Now, $\omega^{g^i}=\{W^{xy^{1+2t}},U^{xy^{1+2t}}\}=\omega =\{W,U\}$ and hence $U=W^{xy^{1+2t}}$, a contradiction. This yields that $\langle g\rangle$ induces a regular orbit on $\omega$. In particular, we may assume that every subspace $U$ of $V$ with $\{W,U\}\in \Omega$ is of the form $W^{xy^{i}}$, for some $i\geq 0$. 
This yields that every subspace $U$ of $V$ with $\{W^y,U\}\in\Omega$ is of the form $(\{W^{xy^i}\mid i\geq 0\})^y=\{W^{xy^{i}}\mid i\geq 0\}$. When $\Omega=\Omega_{k,\perp}$, it follows that $V=W\oplus U$ if and only if $V=W^y\oplus U$, and when $\Omega=\Omega_{k,\leq}$, it follows that $W\leq U$ and $\dim U=n-k$ if and only if $W^y\leq U$ and $\dim U=n-k$. In both cases, elementary geometric considerations imply that $W=W^y$. Since $y$ induces a regular orbit on $\{W^h\mid h\in \langle y\rangle\}$, we get $y=1$ and $g=x$ has order $2$. Now, clearly $g$ has a regular cycle.
\end{proof}

\section{{\bf Case}~$(vii)$ of Proposition~\ref{reduction}: Novelties}\label{bloodynovelties}

Here we may assume that $G_{\alpha} \cap G_0$ is nonmaximal in $G_0$ and since $G$ is a subspace subgroup, the only maximal subgroups of $G_0$ containing $G_{\alpha} \cap G_0$ are either reducible subgroups (as in Definition~\ref{def}(a)) or of type $O_{2m}^{\pm}(2^f)$ when $G_0 = \Sp_{2m}(2^f)'$ (as in Definition~\ref{def}(b)). The case of $G_0 = \POmega_8^{+}(q)$ is special and the novelties have been obtained explicitly by Kleidman~\cite{D4max}. We postpone this case until the end of this section and assume for now that $G_0 \ne \POmega_8^{+}(q)$. Moreover, by Theorem~\ref{smalldimension}, we may assume that $n \ge 5$ and in particular, that $G_0 \ne \PSp_4(q)$.

 We use the notation of~\cite{KL} quite liberally in this section. In particular, we let $H = G_\alpha$ be a maximal subgroup of $G$. By Aschbacher's theorem, $H$ (which is sometimes denoted by $H_G$ in~\cite{KL}) belongs to one of the nine collections $\C_i(G)$ of subgroups of $G$ (where $i \in \{1,\ldots, 9\}$). We define $\Omega$, $\Gamma$ and $A$ depending on $G_0$ according to Table~\ref{OmegaTable}. The reader will notice an abuse of notation here, since we have hitherto used $\Omega$ to denote a primitive $G$-set. For the remainder of this section, we will denote this primitive $G$-set by $\OmSet$. 
 
\begin{table}
\begin{tabular}{|c|c|c|c|} 
\hline
$G_0$ & $\Omega$ & $\Gamma$ & $A$ \\
\hline
 $\PSL_n(q)$ & $\mathrm{SL}_n(q)$ & $\Gamma \mathrm{L}_n(q)$ & $\Gamma \langle \iota \rangle$ \\ 
 $\PSU_n(q)$& $\mathrm{SU}_n(q)$ & $\Gamma \mathrm{U}_n(q)$ & $\Gamma$ \\ 
 $\PSp_n(q)$ & $\mathrm{Sp}_n(q)$ & $\Gamma \mathrm{Sp}_n(q)$ & $\Gamma$ \\ 
 $\POmega_n^{\varepsilon}(q)$& $\Omega^{\varepsilon}_n(q)$ & $\Gamma \mathrm{O}^{\varepsilon}_n(q)$ & $\Gamma$ \\
 \hline
 \end{tabular} \caption{Notation for Section~\ref{bloodynovelties} (from \cite{KL})}
 \label{OmegaTable}
 \end{table}

 For a group $X$ such that $\Omega \le X \le A$ we define $\overline{X}$ to be the quotient group $\overline{X} = X / (X \cap Z)$, where $Z = \Zent {\GL_n(q_0)}$ (and where $q_0$ is as in Notation~\ref{generalnotation}). In particular $\overline{\Omega}$ is the simple group $G_0$ ($G_0 \ne \Sp_4(2)'$ since $n \ge 5$) and we have $G = \overline{X}$ for some $X$ as above. The subgroups $H \in \C_9(G)$ are almost simple, and each has a socle whose covering group $L$ acts absolutely irreducibly on the natural module $V$ and satisfies various other conditions to prevent $H$ from being naturally contained in a subgroup in $\displaystyle \cup_{i=1}^{8} \mathcal{C}_i(G)$. For these groups we define the corresponding subgroup of $\overline{\Omega}$ to be $H_{\overline{\Omega}} = H \cap \overline{\Omega}$.

 For $i=1,\ldots, 8$, we define the subgroup $H_{\Gamma} \in \C_i(\Gamma)$ to be the stabilizer in $\Gamma$ of some geometrical structure on $V$. For example, when $i=2$, we have $V= V_1 \oplus \cdots \oplus V_k$ where $k = \dim V_j$ for all $j$ and $H_{\Gamma} = \norm{\Gamma} { \{ V_1, \ldots ,V_k \}}$. The corresponding subgroup of $\Omega$, denoted $H_{\Omega}$, is then defined to be $H_{\Gamma} \cap \Omega$. Similarly, for $X$ as above, we define the corresponding subgroup $H_X$ of $X$ to be 
 \[ H_X = \begin{cases}
 H_{\Gamma} \cap X & \text{ if } X \le \Gamma; \\
 \norm {X}{H_{\Gamma}} & \text{ otherwise.}
 \end{cases}\]
 Further, we define $H_{\overline{X}}$ to be $H_X / (H_{X} \cap Z)$. In particular $H_{\overline{\Omega}} = H_{\overline{\Gamma}} \cap \overline{\Omega}$. Moreover, in almost all situations, we have $H_{\overline{\Omega}} = H_G \cap \overline{\Omega}$. (Indeed, this is clear from the definitions above if $G \le \overline{\Gamma}$ since $H_{\overline{\Omega}} = H_{\overline{\Gamma}} \cap \overline{\Omega} $ and $H_G = H_{\overline{\Gamma}} \cap G$. Otherwise $G_0$ is linear, $G \not \le \overline{\Gamma}$ and we have $H_G \cap \overline{\Omega} = \norm G{H_{\overline{\Gamma}}} \cap \overline{\Omega} = \norm {\overline{\Omega}} {H_{\overline{\Gamma}}}.$ As $H_{\overline{\Omega}} = H_{\overline{\Gamma}} \cap \overline{\Omega}$, we have $H_{\overline{\Omega}} \le H_G \cap \overline{\Omega}$; however it may occur that equality does not hold when (and only when) $H_G$ is a type $\GL_1(q) \wr S_n$ group. Moreover this only happens when $q=2$ and $n$ is even or when $(n,q) = (2,5), (4,3)$, or $(3,4)$ (see~\cite[ Proposition~3.1.3]{KL}). For example, when $q=2$, $n$ is even and $G= \GL_n(2) \langle \iota \rangle$, we have $H_{\overline{\Omega}} \cong \GL_1(2) \wr S_n \cong S_n$ but $H_G \cap \overline{\Omega} = \norm {\overline{\Omega}} {H_{\overline{\Gamma}}} = 2 \times S_n$.)

 Usually, $H_{\overline{\Omega}}$ is maximal in $G_0$, in which case $H$ is not a novelty and hence $H$ does not concern us. Thus we may assume that $H_{\overline{\Omega}}$ is strictly contained in some subgroup $K_{\overline{\Omega}} \in \C_i(\overline{\Omega})$, for some $i \in \{1, \ldots, 9\}$. 
 It is very important to observe that, since $H=H_G$ is a subspace subgroup, we have $$K_{\overline{\Omega}} \in \C_1(\overline{\Omega}),\,\,
 \textrm{ or  }\, G_0 = \Sp_{n}(2^f) \textrm{ and } K_{\overline{\Omega}} \in \C_8(\overline{\Omega}).$$ 

Complete information on the novelties for $n \ge 13$ is contained in~\cite[Table~3.5.H]{KL}. Reading off the possibilities for $H_{\overline{\Omega}}$ and $K_{\overline{\Omega}}$ in~\cite[Table~3.5H]{KL} we have
\begin{itemize}
\item (row 6) $H_{\overline{\Omega}}$ of type $P_{k,n-k}$, $G_0 = \PSL_n(q)$ and $G$ contains graph automorphisms. This case has been eliminated already in Section~\ref{casev}.
\item (row 7) $H_{\overline{\Omega}}$ of type $\GL_k(q)\oplus \GL_{n-k}(q)$,
 $G_0 = \PSL_n(q)$ and $G$ contains graph automorphisms. This case also has been eliminated already in Section~\ref{casev}.
 \item (row 12) $H_{\overline{\Omega}}$ of type $P_{n/2-1}$,  $G_0$ orthogonal. This case is eliminated in Section~\ref{sec:line12+14}.
\item (row 14) $H_{\overline{\Omega}}$ of type $O_2^+(3) \perp O_{n-2}(3)$, $G_0 = \POmega_{n}^{\pm}(q)$ ($n$ even) and $G$ contains similarities. This case is eliminated in Section~\ref{sec:line12+14}.
\item (row 22) $H_{\overline{\Omega}}$ of type $\GL_{n/2}(q).2$, $G_0$ orthogonal. This case is eliminated in Section~\ref{sec:line22}. 
\end{itemize}
We note that the examples in row $15$ of~\cite[Table~3.5H]{KL}, where $H_{\overline{\Omega}}$ of type $O^{+}_2(2) \wr S_{n/2}$ is contained in $K_{\overline{\Omega}}$ of type $P_{n/2}$, do not in fact yield maximal subgroups of $G$ since $\norm G{H_{\overline{\Omega}}} < \norm {G}{K_{\overline{\Omega}}}$ (see~\cite[p. 68]{KL} and row 12 of~\cite[Table~3.5H]{KL}). 

When $n \le 12$ we need to work  harder to obtain the possibilities for $G$ and $H$.
Again we recall that Aschbacher's theorem~\cite{Asch} asserts that since $H=G_\alpha$ is maximal in $G$, it must be a subgroup of type $\mathcal{C}_i(G)$ for some $i \in \{1, \ldots, 9\}$. 
 Since $G_0 \cap H$ is not maximal in $G_0$, it does not appear in the list of maximal subgroups of $G_0$ in~\cite{kthesis}. (Observe, however, that the maximal subgroups of $G_0$ for $n =12$ are not listed in~\cite{kthesis}.) Loosely speaking, for $n \le 12$, we adopt  (suitably modified) the arguments in~\cite{KL} for our purposes. Moreover these arguments will simplify the checking of the tables of maximal subgroups in~\cite{kthesis}. 

 \begin{proposition} \label{10.1}
 Suppose that $n \ge 5$,  $G_0 \ne 
 \POmega^{+}_8(q)$, $H=G_\alpha$ is a subspace subgroup as in Definition~$\ref{def}$ and that $G_0\cap H$ is not maximal in $G_0$ (so that $H$ is a $G$-novelty). Then $H$ is not a subgroup of type $\mathcal{C}_i(G)$ for any $i \in \{4,\ldots, 9\}$. 
 \end{proposition}
 \begin{proof} 
 First suppose that $H \in \mathcal{C}_4(G)$ so that $H_{\Gamma}$ stabilizes a tensor decomposition $V= V_1 \otimes V_2$ as described in~\cite[Section~4.4]{KL}. We claim that $H_{\Omega}$ acts irreducibly on $V$. For $H_{\Omega}$ contains a subgroup of the form $\Omega_1 \otimes \Omega_2$ (by~\cite[(4.4.13)]{KL}), which acts absolutely irreducibly on $V$ provided $\Omega_i$ acts absolutely irreducibly on $V_i$ for $i=1$ and $2$ by~\cite[Proposition~4.4.3(vi)]{KL}. This is true as long as $\Omega_i \ne \Omega^{\pm}_2(q)$ by~\cite[Proposition~2.10.6]{KL}; but $\Omega_i \ne \Omega^{\pm}_2(q)$ in all cases by definition of $\mathcal{C}_4$~\cite[Table~4.4.A]{KL}. Thus $H_{\Omega}$ acts irreducibly on $V$ and hence $H$ cannot be a subspace subgroup as in Definition~\ref{def} (note that when $G_0 = \Sp_n(2^f)$ there are no $\mathcal{C}_4$-subgroups~\cite[Table~4.4A]{KL}). 

 Similarly suppose that $H \in \mathcal{C}_7(G)$ so that $H_{\Gamma}$ stabilizes a subspace decomposition $V= V_1 \otimes \cdots \otimes V_t$, with $\dim V_i = k$ and $n = \dim V = k^t$. Again there are no $ \mathcal{C}_7$-subgroups when $G_0 = \Sp_n(2^f)$ so it suffices to show that $H_{\Omega}$ acts irreducibly on $V$. Now $H_{\Omega}$ contains a subgroup $L$ of the form $\Omega_1' \otimes \cdots \otimes \Omega_t'$, and  $\Omega_i'$ acts absolutely irreducibly on $V_i$ by~\cite[Corollary~2.10.7]{KL} (each $\Omega_i'$ is quasisimple by definition of $\mathcal{C}_7$~\cite[p. 156]{KL}). Thus $L$ and $H_\Omega$ act absolutely irreducibly on $V$ by~\cite[Lemma~4.4.3(vi)]{KL}. 

 Next, suppose that $H$ is a subfield subgroup in $\mathcal{C}_5(G)$. Now $H_\Omega$ contains a subgroup $\Omega_{\sharp}$ by~\cite[(4.5.5)]{KL}, which acts absolutely irreducibly on the corresponding vector space $V_{\sharp}$ over the corresponding subfield by~\cite[Proposition~2.10.6]{KL}, and hence $\Omega_{\sharp}$ and $H_{\Omega}$ act irreducibly on $V = V_\sharp \otimes \mathbb{F}_{q}$. Thus we may assume that $q$ is even, $G_0$ is symplectic and $H_{\Omega}$ preserves a quadratic form $Q$. However this is absurd. Indeed, let $w \in V_\sharp \setminus \{0\}$ with $Q(w)=0$. Now $\Omega_{\sharp} = \Sp_{n}(q_0)$ is transitive on the non-zero vectors of $V_\sharp$ by~\cite[Lemma~2.10.5]{KL}, which implies that $Q(v)=0$ for all $v \in V_\sharp$. It follows that the associated bilinear form $f_Q$ vanishes on $V$, which is a contradiction since $f_Q$ is non-degenerate (by definition). Thus $H_{\Omega}$ is irreducible and not contained in a type $O_{n}^{\pm}(2^f)$ subgroup when $G_0 = \Sp_{n}(2^f).$

Now suppose that $H \in \mathcal{C}_6(G)$. Note that if $G_0$ is symplectic then $q$ is odd (by~\cite[Table~3.5.C]{KL}), so it suffices to show that $H_{\Omega}$ acts irreducibly on $V$. Certainly the $r$-group $R$ in the definition of $\mathcal{C}_6(\Gamma)$ acts (absolutely) irreducibly on $V$ by definition~\cite[p.150]{KL}. Moreover, $R$ is contained in $H_\Omega$ (see the proof of~\cite[Proposition~4.6.4]{KL}) and hence $H_{\Omega}$ also acts irreducibly on $V$.

Now let us assume that $H$ is of classical type in $\mathcal{C}_8(G)$ so that $G_0$ is either linear or symplectic (see~\cite[Table~4.8.A]{KL}). In all these cases, since $n \ge 3$, $H_{\Omega}$ acts absolutely irreducibly on $V$ by~\cite[Proposition~2.10.6]{KL}. Thus we may assume that $G_0 = \Sp_{n}(2^f)$ and that $H$ is of type $O_{n}^{\pm}(q)$. We claim that these groups do not give rise to novelties. When $n\le 10$,  $G_0 \cap H$ is maximal in $G_0$ for all $q$  by~\cite{kthesis}. When $n =12$, the argument of~\cite[Proposition~7.8.1]{KL} shows that $H_{\Omega}$ is not contained in a subgroup of type $\mathcal{C}_i(\Omega)$ for $i = 1,\ldots,8$. Furthermore, it is easy to check that $H \cap G_0$ is much too big to be contained in a $\mathcal{C}_9$-subgroup since $\dim O_{12}( \overline{\mathbb{F}}_{q}) = 66$ and so $|H|$ is greater than $\max \{q^{3n}=q^{36}, (n+2)!=14!\}$ (see~\cite[Theorem~5.2.4]{KL}). Thus $H \cap G_0$ is always maximal in $G_0$. 

Finally, suppose that $H$ is an almost simple subgroup of type $\mathcal{C}_9$. Then, by definition, $H_\Omega$ acts absolutely irreducibly on $V$~\cite[p.3]{KL}; thus we may assume that $G_0 = \Sp_{n}(2^f)$  and $H \cap G_0$ is only contained in a type $O_{n}^{\pm}(2^f)$ subgroup. In particular, the socle of $H _{\Omega}$ fixes a quadratic form, which cannot happen when $G_0 = \Sp_{n}(2^f)$ by~\cite[p.3]{KL}.
 \end{proof}
The next two lemmas prove results for $\C_2$-subgroups and $\C_3$-subgroups. 
 \begin{lemma} \label{10.2}
 Suppose that $n \ge 5$, that $G_0 \ne 
 \POmega^{+}_8(q)$, that $H=G_\alpha$ is a subspace subgroup as in Definition~$\ref{def}$ and that $G_0 \cap H$ is not maximal in $G_0$ (so that $H$ is a $G$-novelty). Then $H$ is not a $\C_2$-subgroup of type $\GL^{\pm}_k(q) \wr S_t$, $\Sp_k(q) \wr S_t$ or $O_k^{\zeta}(q) \wr S_t$. 
 \end{lemma}
 \begin{proof} 
We argue by contradiction and we assume that  $H$ is a $\C_2$-subgroup of type $\GL^{\pm}_k(q) \wr S_t$, $\Sp_k(q) \wr S_t$ or $O_k^{\zeta}(q) \wr S_t$. Here $H_{\Gamma}$ is the stabilizer $\norm {\Gamma}{\{V_1, \ldots, V_t\}}$ where $V = V_1 \oplus \cdots \oplus V_t$ is a decomposition $\mathcal{D}$ of $V$ into subspaces of equal dimension $k$. Observe that we are assuming that the $V_i$s are non-degenerate when $G_0$ is not linear. 

Observe that when $q$ is even and $G_0$ is symplectic we have $q\geq 4$, since otherwise $\mathrm{Out}(G_0)=1$ and there are no novelties.

 We claim that subgroups $H$ of type $\Sp_k(2^f) \wr S_{t}$ do not fix a quadratic form $Q$ (when $q \ge 4$).  In fact, if this is not the case, then $H_\Omega\leq K$ with $K=O_n^{\pm}(q)$. Now, note that $\Sp_k(2^f)^{t}$ is contained in $H_\Omega$. As $\Sp_k(2^f)^{t}$ acts irreducibly on $V_1$, we see that $V_1$ is either totally singular or non-degenerate with respect to $Q$. Suppose that $V_1$ is totally singular. Let $g \in \Sp_k(2^{f})$ with $g \ne 1$ and set $g_1=(g,1, \ldots,1) \in \Sp_k(2^{f})^t$. Clearly, $g_1$ acts nontrivially on $V_1$ and trivially on $V_i$ for $i \ge 2$. Therefore $g_1$ acts nontrivially on $V_1$ and trivially on $V/V_1^{\perp_Q}$. However,  this contradicts~\cite[Lemma~4.1.9]{KL}: if $g_1$ acts nontrivially on $V_1$, then it  acts nontrivially also on $V/ V_1^{\perp_Q}$. This shows that  $V_1$ is non-degenerate with respect to $Q$ and hence the natural image $K_{V_1}$ of $K$ in $\GL(V_1)$ is of type $O_k^{\pm}(2^{f})$. Now observe that $K_{V_1}$ contains the natural image $H_{V_1}$ of $H_{\Omega}$ in $\GL(V_1)$ and that $H_{V_1}$ is of type $\Sp_k(2^f)$. This can only be true if $k=q=2$, which contradicts our assumption that $q \ge 4$. Thus $\Sp_k(2^f) \wr S_{t}$ does not fix a quadratic form (when $q\ge 4$).

Now the result follows by showing that $H_{\Omega}$ acts irreducibly on $V$. 
 By~\cite[Proposition~2.10.13]{KL},  $H_{\Omega}$ acts irreducibly on $V$ if the following three conditions hold: 
\begin{itemize}
\item[(i)] $(H_{\Omega})_{(\mathcal{D})}$, the subgroup of $H_{\Omega}$ fixing setwise each $V_i$, acts irreducibly on each $V_i$; 
\item[(ii)] the $V_i$ are pairwise non-isomorphic $(H_{\Omega})_{(\mathcal{D})}$-modules; and 
\item[(iii)] $H_\Omega$ acts transitively on $\{V_1, \ldots V_t\}$.
\end{itemize}
 To prove (i), we may assume that $\dim V_i \ge 2$, otherwise the result is trivial. By~\cite[Proposition~2.10.6]{KL},  (i) holds  unless $\Omega_i = \Omega(V_i)$ is of type $\Omega_2^{\pm}(q)$, since $(H_{\Omega})_{(\mathcal{D})}$ contains $\Omega_i$.
 In this exceptional case, $G_0$ is orthogonal and $n$ is even so we may assume that $ n \ge 8$ and $t \ge 4$. Now by~\cite[(6.2.3)]{KL}, for each $g' \in O_m^{\zeta}(V_i)$, there exists $g \in (H_{\Omega})_{(\mathcal{D})}$ that acts on $V_i$ in the same way as $g'$. In particular, $(H_{\Omega})_{(\mathcal{D})}$ acts absolutely irreducibly on $V_i$ in this case as well, unless $q=2,3$ and $\Omega_i(q) = \Omega_2^{+}(q)$ by~\cite[ Proposition~2.10.6(iii)]{KL}.

 For (ii) we simply observe that if $\Omega_i = \Omega(V_i)$ is nontrivial, then $(H_{\Omega})_{(\mathcal{D})}$ contains an element $g$ that acts nontrivially on $V_i$ and trivially on $V_j$ for all $ j \ne i$. So (ii) holds unless $\Omega_i =1$, which is only true if $\Omega_i = \mathrm{SL}_1(2)$, $\Omega_2^{+}(2) $ or $\Omega_2^{+}(3)$. 


Finally (iii) always holds from~\cite[Lemma~4.2.2(iii)]{KL} (see also Proposition~4.2.9, p.106, p.109 and (4.2.12) of~\cite{KL}).

 The lemma now follows from~\cite[Proposition~2.10.13]{KL} unless $H$ is of type $$\GL_1(2) \wr S_n,\,\, O^{+}_2(2) \wr S_{n/2}\,\, \textrm{or}\,\, O_2^{+}(3) \wr S_{n/2}.$$ 

 In the first case,  $H$ stabilizes the decomposition $V = V_1 \oplus \cdots \oplus V_n$ where $V_i = \langle v_i \rangle$ and $(v_1,\ldots,v_n)$ is a basis of $V$.  If $n$ is odd, then $H$ is contained in $\norm G K$ where $K$ is a subgroup of type $\GL_1(2) \oplus \GL_{n-1}(2)$; therefore this case does not occur. Suppose $n$  even, define $f(v_i,v_j) = 1 + \delta_{ij}$ and extend $f$ bilinearly to $V$. It is easy to show that $f$ is a symmetric non-degenerate bilinear form on $V$; thus $H_{\Omega}$ is contained in $\Sp_{n}(2)$ and $H$ is not a subspace subgroup. 

 In the second case, subgroups $H$ of this type are never maximal, for they are contained in  $\norm G K $ where $K$ is a subgroup of type $P_{n/2-1}$ (see~\cite[p. 68]{KL}); so these cases do not concern us. 

In the third case, write $V_i = \langle e_i,f_i \rangle$ with $e_i+f_i$  non-degenerate, $(e_i+f_i,e_i+f_i)=-1$ (a non-square when $q=3$) and $Q(e_i-f_i) = +1$ (a square). Write $$W_1 = \langle e_1+f_1, \ldots , e_{n/2}+f_{n/2}\rangle\,\, \textrm{and}\,\, W_2 = \langle e_1-f_1, \ldots , e_{n/2}-f_{n/2}\rangle.$$ Note that $H_{\Omega}$ fixes the orthogonal decomposition $W_1 \perp W_2$. Since the discriminant on $\langle e_i+f_i \rangle$ is $D(Q_{\langle e_i+f_i \rangle}) = -1$ and the discriminant on $ \langle e_i-f_i \rangle$ is $D(Q_{\langle e_i-f_i \rangle}) = 1$ (\cite[(2.5.14)]{KL}), we have $D(W_1) = (-1)^{n/2}$ and $D(W_2) = 1$ by~\cite[ Proposition~2.5.11]{KL}. When $n/2$ is odd, the discriminant $D(W_1)$ is a non-square and the discriminant $D(W_2)$ is a square, and hence $W_1$ and $W_2$ are non-isometric $(n/2)$-subspaces. Therefore $H_{\Omega}$ is contained in a $\C_2$-subgroup of type $O_{n/2}(q)^{2}$. If $n/2=2a$ is even then $D(W_i)=+1$ and $sgn(Q_{W_i}) = +$ if and only if $a(q-1)/2 = n(q-1)/8$ is even by~\cite[Proposition~2.5.10]{KL}. So $H_{\Omega}$ is contained in a subgroup of type $O_{n/2}^{\pm}(q) \wr S_2$. Thus in all cases $H_\Omega$ is contained in another $\C_2$-subgroup of $\Omega$ and therefore $H$ is not a subspace subgroup. 
\end{proof}

\begin{lemma} \label{10.3}
 Let $H$ be a $\C_3$ field extension subgroup of $G$ of type $X_t(q^{r})$ and suppose that $G_0 \ne 
 \POmega_8^{+}(q)$ and $n \ge 5$. Then $H_{\Omega}$ acts irreducibly on $V$ when $t \ge 2$ (which is always true when $G_0$ is symplectic or orthogonal). If $G_0$ is linear or unitary, then $H_{\Gamma} \cap \GL^{\pm}_n(q)$ acts irreducibly on $V$ in all cases.
 \end{lemma}
 \begin{proof} 
 See~\cite[Lemma~4.3.2 and Proposition~4.3.3(i)]{KL}.
 \end{proof}

 \begin{proposition}\label{willthiseverend} 
 Suppose that $n \ge 5$, that $G_0 \ne 
 \POmega^{+}_8(q)$, that $H=G_\alpha$ is a subspace subgroup as in Definition~$\ref{def}$ and that $G_0 \cap H$ is not maximal in $G_0$ (so that $H$ is a $G$-novelty). Then $H$ is one of the generic examples in rows $6$, $7$, $12$, $14$, $22$ of~\cite[Table~3.5H]{KL}.
 \end{proposition}
 \begin{proof} 
 By Proposition~\ref{10.1} and Lemmas~\ref{10.2} and~\ref{10.3}, for $n \le 11$ we just need to check in~\cite{kthesis} that subgroups of $G_0$ of the following types are indeed maximal in $G_0$: 
 \begin{enumerate}
 \item type $\C_1 (\overline{\Omega})$,
 \item type $\GL_1^{\epsilon}(q^{n})$ when $n$ is prime and $G_0 = \PSL_n^{\epsilon}(q)$,
 \item type $\GL_{n/2}(q^{u}).2$ when $G_0 = \mathrm{PSU}_n(q)$, $\POmega_{2m}^{+}(q)$ or when $G_0 = \PSp_{2m}(q)$ and $q$ is odd,
 \item type $O_{n/2}(q)^{2}$ when $q$ is odd and $n/2$ is odd. 
 \end{enumerate}
 Note that when $G_0$ is symplectic, we may assume that $q \ge 3$ since otherwise $\mathrm{Out}(G_0)=1$ and there are no novelties.
 The result now follows easily by inspection unless (i) $G_0 = \PSU_6(2)$, $\PSU_5(2)$, 
 or (ii) $G_0 = \POmega^{\pm}_{n}(2)$ or $\POmega_{n}^{\epsilon}(3)$ and $H$ is a $\C_1$-subgroup of type $O_2^{+}(q) \perp O_{n-2}^{\epsilon}(q)$. We check  the case $G_0 = \PSU_6(2)$ and $\PSU_5(2)$ with a comutation in \texttt{magma}, and we will deal with case (ii) together with the case $n=12$.

 When $n =12$, Proposition~\ref{10.1} and Lemmas~\ref{10.2},~\ref{10.3} show that there are no novelties with $H$ a group of type $\C_i$ for $i \in \{2,\ldots, 9\}$. For reducible subgroups, it is straightforward to show that the only possibilities are either in rows 6, 7, 12, 14, 22 of~\cite[Table~3.5H]{KL}, or $G_0 = \POmega_{12}^{\pm}(2)$ and $H$ is of type $O_2^{+}(2)\perp O^{\epsilon}_{10}(2)$. 

 But in the latter case, $\Aut(G_0) = G= \mathrm{PSO}_{2m}^{\pm}(2)$, and $H$ stabilizes an orthogonal decomposition $V = V_1 \perp V_2$ with $\dim V_1=2$. There is a unique non-singular vector $v \in V_1$, thus $H$ and $H_{\Omega}$ also stabilize the decomposition $V = \langle v \rangle \perp \langle v \rangle^\perp$. That is, $H$ and $H_{\Omega}$ are contained in subgroups of type $\Sp_{n-2}(2)$ and no novelties occur. 

Finally, suppose $G_0= \POmega_{n}^{\epsilon}(3)$ and $H$ is a $\C_1$-subgroup of type $O_2^{+}(q) \perp O_{n-2}^{\epsilon}(q)$. The case $n$ even is  in row 14 of~\cite[Table~3.5H]{KL} and hence we may assume that $n$ is odd and $G= \mathrm{PSO}_{n}(3)$. Now $H$ and $H_{\Omega}$ stabilize an orthogonal decomposition $V = V_1 \perp V_2$ where $\dim V_1=2$; it is easy to see that $V_1$ has only one $1$-space $\langle v \rangle$ such that $Q(v)=1$ ($v = \pm (e_1-f_1)$) and only one $1$-space $\langle w \rangle$ such that $Q(w)=-1$ (which is a nonsquare when $q=3$). Thus $G_\alpha$ fixes $\langle v \rangle$ and $\langle w \rangle$ 
 and hence $H$ is contained in a subgroup of type $O_1(3) \perp O^{\pm}_{n-1}(3)$; so there are no novelties here.
 \end{proof}

Proposition~\ref{willthiseverend} shows that, when $G_0\neq \POmega_8^+(q)$, for the proof of Theorem~\ref{thrm:AS} the only $G$-novelties we need to consider are the generic examples in rows $6$, $7$, $12$, $14$, $22$ of~\cite[Table~3.5H]{KL}.

Next, we turn our attention to $G_0 = \POmega_8^{+}(q)$. We check~\cite{D4max} carefully and we find that the only possibilities are the following: 
\begin{itemize}
\item[(i)] $H \cap G_0$ is the stabilizer of a totally singular $3$-subspace and $G$ does not contain a triality automorphism (as in row 12 of~\cite[Table~3.5H]{KL}); 
\item[(ii)] $G$ does contain a triality automorphism and $H$ is a parabolic subgroup $P_2$ corresponding to the removal of the three outer nodes of the Dynkin diagram. 
\end{itemize}

For the proof of Theorem \ref{thrm:AS}, case (i)  is considered in Section~\ref{sec:line12+14}. We conclude this section dealing with case (ii). Here we know that $|H|$ divides $q^{11}(q-1)^2
|\PGL_2(q)||\Sym(4)|e$; thus for $x \in G$ of prime order $r$ we have $\fpr_{\Omega}(x)=0$ unless $r \mid 6ep(q^2-1)$, by Lemma~\ref{obvious3}. Now the usual calculations show that
\[
\SGO\leq \omega(6ep(q^2-1))\frac{4}{3q},
\]
 which is less than $1$ unless $q=2,4$. We eliminate the case $q\in \{2,4\}$ with the help  of \texttt{magma}. 

\subsection{Novelties: Lines~$12$  and~$14$ in \cite[Table~$3.5H$]{KL}}\label{sec:line12+14}
 Here $G_0=\POmega_{2m}^+(q)$ and $G_0\lhd G\leq \mathrm{Aut}(G_0)$. The group $G$ acts primitively on the set $\OmSet$ of totally isotropic subspaces of $V$ of dimension $m-1$ or on the set $\OmSet$ of non-degenerate $2$-subspaces of $+$ type. 
\begin{proposition}\label{line12+14}
If $g\in G$, then $g$ has a regular orbit on $\OmSet$.
\end{proposition}

\begin{proof}
By Theorem~\ref{FGK}, the permutation character of $G$ in its action on $\OmSet$ is contained in the permutation character of  $G$ in its action on the totally singular $1$-subspaces of $V$. Now the result follows from Lemma~\ref{characterlemma} and Proposition~\ref{caseiiiPSLPSU}.
\end{proof}

\subsection{Novelties: Line~$22$ in \cite[Table~$3.5H$]{KL}}\label{sec:line22}

Here $G_0=\POmega_{2m}^+(q)$ with $m$ odd and $G_0\lhd G\leq \mathrm{Aut}(G_0)$. The group $G$ acts primitively on the set 
\[\OmSet=\{\{U,W\}\mid V=U+W,\,\dim U=\dim W=m,\,\textrm{and }U,W\textrm{ totally singular}\}.\] 
\begin{proposition}\label{line22}
If $g\in G$, then $g$ has a regular orbit on $\OmSet$.
\end{proposition}
\begin{proof}
The proof is similar to the proof of Proposition~\ref{propPSLgraph}. From Remark~\ref{smallorth}, we have $m\geq 4$, and hence $m\geq 5$ because $m$ is odd. 

 From Proposition~\ref{totiso}, the element $g$ has a regular orbit in its action on the totally singular subspaces of $V$ of dimension $m$. Let $U$ be an element in such a regular orbit. Suppose that there exists $W\notin \{U^{g^t}\mid t\geq 0\}$ with $\{U,W\}\in \OmSet$. Then, clearly, $g$ has a regular orbit on $\{U,W\}$. In particular, we may assume that the only totally singular subspaces $W$ of $V$ with $V=U+W$ are of the form $U^{g^i}$, for some positive integer $i$. An easy computation shows that there are exactly $q^{m(m-1)/2}$ totally singular complements $W$ for $U$ in $V$. Therefore $|g|\geq q^{m(m-1)/2}$. However this contradicts~\cite[Theorem~$2.16$]{GMPS}, which states that $|g|\leq q^{m+1}/(q-1)$.
\end{proof}

\end{document}